\documentclass[12pt]{report}
\usepackage{amsfonts}
\usepackage{amssymb}
\usepackage{amsthm}
\usepackage{amsmath}
\usepackage{appendixhack}

\newtheorem{theorem}{Theorem}

\newtheorem{corollary}[theorem]{Corollary}

\newtheorem{definition}{Definition}

\newtheorem{remark}{Remark}
\newtheorem{lemma}{Lemma}

\newcommand{\thesistitle}{Quenched Large Deviations for Multidimensional Random Walk in Random Environment: A Variational Formula}
\newcommand{\myname}{Jeffrey M. Rosenbluth}
\newcommand{\advisor}{S.R.S. Varadhan}
\newcommand{\thesismonth}{January,}
\newcommand{\thesisyear}{2006}

\newcommand{\esup}{\operatornamewithlimits{ess\, sup}}
\newcommand{\wlim}{\operatornamewithlimits{w-lim}}
\newcommand{\lb}{\left \langle}
\newcommand{\rb}{\right\rangle}
\newcommand{\mpt}{T}
\newcommand{\admiss}{\mathcal{K}}
\newcommand{\unit}{e}
\newcommand{\base}{\mathrm{e}}
\begin{document}

\linespread{1.6}   %double spaced
\normalsize        %set to normalsize so double spacing kicks in

\pagenumbering{roman}  % front matter number with roman

%%-------------------------------------Title page (unnumbered)
\thispagestyle{empty}
\vspace*{0.5cm}
\begin{center}
\Large\bfseries 
\thesistitle
\end{center}

\begin{center}
\vspace{1cm}
by
\vspace{1cm}

{\it \large \myname}\\

\vspace*{4cm}

\noindent
A dissertation submitted in partial fulfillment\\
of the requirements for the degree of \\
Doctor of Philosophy\\
Department of Mathematics\\
New York University\\
\thesismonth~\thesisyear\\
\end{center}

\vspace{1cm}

\begin{flushright}
   \makebox[2in]{\hrulefill}\\
   \vspace{-.2cm}
   \parbox[t]{1.8in}{\advisor}\\
\end{flushright}
\newpage

%%-------------------------------------Copyright  (unnumbered)
%%                                  (only if choose to copyright)
\thispagestyle{empty}
\vspace*{2in} 
\begin{center}
\copyright \hspace{.2cm} \myname\\
All Rights Reserved, \thesisyear
\end{center}
\newpage
%\thispagestyle{empty} \hbox{   } \newpage

%%-------------------------------------BLANK (unnumbered)
%\thispagestyle{empty} \hbox{   } \newpage

%%-------------------------------------Frontispiece (optional, unnumbered)
%\include{frontispiece}

%%-------------------------------------Dedication, (first numbered)
%\pagestyle{plain}
%\vspace*{3in} 
%\begin{center}
%{\large Dedicated to ...}
%\end{center}

%\addcontentsline{toc}{chapter}{Dedication}
%\hbox{   } \newpage

%-------------------------------------Acknowledgements
\chapter*{Acknowledgements}
\addcontentsline{toc}{chapter}{Acknowledgements}

I would like to express my deepest gratitude to my advisor Professor Srinivasa R.S. Varadhan for superb guidance, understanding, patience, and generosity. His profound insight and brilliant ideas have been essential to this project. I would also like to thank the faculty of The Courant Institute for teaching me mathematics.
\newpage

%%-------------------------------------Preface
%\include{preface}
%\newpage

%%-------------------------------------Abstract

%%%%%%%%%%%%%%%%%%%%%%%%%%%%%%%%%%%%%%%%%%%%%%%%%%%%%%%%%%%%%%% %
%%% uncomment out the following when printing the abstract alone
%%%%%%%%%%%%%%%%%%%%%%%%%%%%%%%%%%%%%%%%%%%%%%%%%%%%%%%%%%%%%%%%
%\begin{center}
%{\bf \thesistitle}\\Author: \myname\\ Advisor: \advisor\\ 
%\bigskip
%{\bf\Large Abstract}\\
%\end{center}
%\medskip
%\thispagestyle{empty}
\chapter*{Abstract}
\addcontentsline{toc}{chapter}{Abstract}

We take the point of view of the particle in a multidimensional nearest neighbor random walk in  random environment (RWRE). We prove a quenched large deviation principle and derive a variational formula for the quenched rate function. Most of the previous results in this area rely on the subbadditive ergodic theorem. We employ a different technique which is based on a minimax theorem. Large deviation principles for RWRE have been proven for i.i.d. nestling environments subject to a moment condition and for ergodic uniformly elliptic environments. We assume only that the environment is ergodic and the transition probabilities satisfy a moment condition.             
\newpage

%%-------------------------------------TOC
\tableofcontents
\newpage

%%-------------------------------------LOF
%\addcontentsline{toc}{chapter}{List of Figures}
%\listoffigures
%\newpage

%%-------------------------------------LOT
%\addcontentsline{toc}{chapter}{List of Tables}
%\listoftables
%\newpage

%%-------------------------------------List of appendices
%%                             (only use if more than one appendix)
%\listofappendices
%\newpage

%%-------------------------------------Thesis body

\pagenumbering{arabic}   %start arabic numbering

\chapter{Introduction}
\section{The Model}
The random walk in a random environment (RWRE) is usually described as a time homogeneous Markov chain (random walk) whose transition probabilities depend on a randomly chosen environment. One can then define an auxiliary Markov chain on the space of environments; this is commonly called ``the environment viewed from the particle''. Here we will take the reverse (but mathematically equivalent) approach and begin by defining a Markov chain on a suitably chosen space of environments, and then consider the ``shadow'' Markov chain on the space $\mathbb{Z}^d$.

We model the environment with a probability space and an ergodic family of commuting measure preserving transformations $(\Omega,\mathcal{F},\mathbb{P},\mpt_\unit)$, $\unit \in U$ and $U=\{\unit : \unit \in \mathbb{Z}^d,  |\unit|=1\}$. By an ergodic family we mean that any set that is invariant under $\emph{all}$  of the $\{T_e\}_{e \in U}$ has measure zero or one. This is less restrictive than the assumption that each $T_e$ is ergodic. The space $\Omega$ is usually taken to be the space of maps $\omega: \mathbb{Z}^d \times U \mapsto [0,1]$, such that $\sum_{\unit \in U} \omega (z ,\unit)=1$ for all $z \in \mathbb{Z}^d$. Here we place no such restriction on $\Omega$, but instead consider the more general case where $\Omega$ is an arbitrary space and $\mathcal{F}$ is a countably generated $\sigma$-algebra. The transformations $\mpt_\unit$ and $\mpt_{ -\unit}$ are inverses of each other. We are also given a map $p:\Omega \times U \rightarrow [0,1]$ such that for all $\omega$ we have $\sum_{e \in U} p(\omega,e)=1$. We construct a transition function, also called $p$, by defining $p(\omega,\mpt_e \omega)=p(\omega,e)$. We fix $\omega$ and consider the Markov chain $\{\overline{\omega}_n\}$ with transition function $p(\omega, \mpt_\unit \omega)$, state space $\Omega$ and induced measure $P_\omega(\overline{\omega}_0=\omega)=1$. Under $P_\omega$ we associate a shadow markov chain $\{X_n\}$ in $\mathbb{Z}^d$ starting at 0, that moves one step in the direction $\unit$ according to which $\mpt_\unit$  is chosen. The requirement $\unit \in U$ makes the RWRE a nearest neighbor random walk. Our main results are a large deviation principle and a variational formula for its rate function for the quenched random walk in random environment.

\section{Notation}
We list some of the notations that will be used throughout.  We denote by $\unit_i$ the vector in $U$ with a 1 in the $i$th coordinate and zeros elsewhere. Since we often need to exponentiate functions of the vectors $e \in U$, we denote the base of the natural logarithm by the roman typeface $\mathrm{e}$ (e.g., we wirte $\mathrm{e}^x$ for $\exp(x)$). For integer vectors $x$ we denote the transition probability on the shadow Markov chain $\{X_n\}$ again by $p$ and write $p(x,x + \unit)=p(\mpt_x\omega,\mpt_{x +\unit} \omega)$, where $\mpt_x$ is the obvious generalization of the transformation $\mpt$ by the vector $x \in \mathbb{Z}^d$. Similarly we will write $F(x,x+ \unit)$ for $F(T_x \omega ,\unit)$. For a function $h(\omega)$ we define the operator $T_\unit h(\omega)=h(T_\unit \omega)$. For $x \in \mathbb{Z}^d$, we take $|x|$ to be the $\ell_1$ norm, that is $|x|=|x_1|+\cdots + |x_d|$. We use $\mathbb{E}$ for expectation with respect to $\mathbb{P}$. The quantities $c,c_0, c_1,  \ldots$ are positive constants and we note that constants may change value from one line to the next. The closed $L^1$ ball centered at $a$ with radius $r$, that is $\{x \in \mathbb{R}^d: |x-a| \leq r\}$ is denoted by $B_r(a)$ and when $a=0$ simply by $B_r$.

\section{Previous Results}
Previous results concerning large deviations for the nearest neighbor RWRE typically rely on the subadditive egodic theorem. In our proofs we use only the multivariate ergodic theorem and rely more heavily on the minimax theorem of Ky Fan \cite{kyfan:53}. 

The first quenched large deviation principle for a multidimensional RWRE is due to Zerner \cite{zerner:98}.  He assumes that the environment is not only ergodic but is i.i.d., that is $\{p(x,x+ \unit)\}_x$ is an i.i.d. family of random variables. He also assumes (as we do) that the transition probabilities satisfy the moment condition,
$$
\int_\Omega (- \log p(\omega ,\unit))^d < \infty
$$
The most limiting restriction of Zerner's result however, is that he proves a large deviation principle only for so called ``Nestling Environments''.
\begin{definition}
A random environment and it's transition function are said to have the nestling property if the convex hull of the support of the law of
$$
\sum_{\unit \in U} p(\omega ,\unit )\unit
$$
contains the origin.
\end{definition}

More recently Varadhan \cite{varad:03}  considers ergodic environments and dispenses with the nestling assumption. He proves both quenched and annealed large deviation principles. He does however restrict the transition functions to be uniformly elliptic, that is
$$
0< a \leq p(\omega ,\unit) \leq b <1
$$ 
with $\mathbb{P}$ probability 1.

We will consider ergodic environments that satisfy a slightly stronger moment condition than that of Zerner, namely
$$
\int_\Omega |\log p(\omega ,\unit)|^{d+\alpha} < \infty
$$
for some $\alpha >0$.

%-----------------------------------------------------------------------------------------------------------------------------------
\section{Results}
The basis for our results is the existence of the logarithmic moment generating function $ \lim _{n \rightarrow \infty} \frac{1}{n} \log E^{P_\omega}\left [\base^{ \lb \lambda, X_n \rb}  \right]$. We then use this to derive the large deviation principle for the RWRE. We define first the class of functions required in the variational formula for the rate function. We will denote by $\admiss$ the class of mean zero functions whose sum over any closed loop is zero and are in $L^{d+\alpha}(\mathbb{P})$ (that is $\mathbb{E}[|F|^{d+\alpha}] \ < \infty$) for some $\alpha >0$.
\begin{definition}
A function $F: \Omega \times U \rightarrow \mathbb{R}$ is in class $\admiss$ if it satisfies the following three conditions:
\begin{enumerate}
 \item[(i)] Moment: for each $\unit \in U$, $F \in \bigcup_{\alpha>0} L^{d+\alpha}(\mathbb{P})$. 
  \item[(ii)] Mean Zero: for each $\unit \in U$, $\mathbb{E}[F(\omega ,\unit)]=0$.
  \item[(iii)] Closed Loop: For any finite sequence $\{x_i\}_{i=0}^n \in \mathbb{Z}^d$, such that $x_{i+1}-x_i \in U$ and $x_0=x_n$ 
  $$\sum_{i=0}^{n-1} F(x_i,x_{i+1})=0$$ 
    \end{enumerate} 
 \end{definition}
\begin{remark} \label{cl}
The closed loop condition $(iii)$ in the above definition guarantees that for any two points $x,y \in \mathbb{Z}^d$ the sum  $\sum_{i=0}^{m-1} F(z_i,z_{i+1})$ where $z_0=x, z_m=y$ and $z_{i+1}-z_i \in U$ is independent of the path $\{z_i\}$ chosen.
\end{remark}
According to Remark \ref{cl}, we can define unambiguously the sum of $F$ from one point $x$ to another $y$. We observe that the path can be chosen so that the number of summands is $|x-y|$. 
\begin{definition}
For $x,y \in \mathbb{Z}^d$ we define 
$$
\sum_{x \rightsquigarrow y} F
$$
as the sum of $F$ over any path from $x$ to $y$ as in Remark \ref{cl}. We also define for $x \in \mathbb{Z}^d$,
$$
f(x)=\sum_{0 \rightsquigarrow x} F
$$
\end{definition}
The large deviation principle will be stated in terms of a function, $\Lambda$.
\begin{definition} \label{lam}
$$
\Lambda(\lambda) := \inf_{F \in \admiss} \esup_\omega \log \sum_{\unit \in U} p(\omega  ,\unit) \base^{\lb \lambda , \unit \rb+F(\omega ,\unit)}
$$
where the $\esup$ is with respect to the measure $\mathbb{P}$.
\end{definition}
\begin{theorem} \label{a}
Suppose $\int \left |\log p(\omega ,\unit)\right|^{d+\alpha} d\mathbb{P}<\infty$ for some $\alpha >0$ and all $\unit \in \mathbb{Z}^d,  |\unit|=1$. Then
\begin{equation}
\begin{split}
 \lim _{n \rightarrow \infty} \frac{1}{n} \log E^{P_\omega}\left [\base^{ \lb \lambda, X_n \rb}  \right]&=\Lambda(\lambda)
 \end{split}
 \end{equation}
\end{theorem}
\begin{theorem}{(Large Deviation Principle)}  \label{b} Under the assumptions of Theorem \ref{a}, 
$X_n/n$ obeys a large deviation principle with rate function
\begin{equation}
I(x)=\sup_\lambda \left\{ \lb \lambda, x \rb - \Lambda(\lambda) \right \}
\end{equation}
\end{theorem}
\begin{remark}
In one dimension $\alpha$ can be taken to be zero.
\end{remark}
\chapter{Functions in class $\admiss$}
In this chapter we prove an important property of functions in class $\admiss$. Theorem \ref{fon} below will play an a key role in the proof of the upper bound in Chapter 3. The main result of this chapter is;
\begin{theorem} \label{fon}
For $F \in \admiss$,
\begin{equation} \label{on}
\lim_{n \rightarrow \infty} \sup_{\substack{|z| \leq n \\ z \in \mathbb{Z}^d}} \frac{f(z)}{n}=0
\end{equation}
\end{theorem}
%-----------------------------------------------------------------------------------------------------------------------------------
In Chapter 3 we will use this theorem to show that given $\epsilon >0$ for $n$ large enough,
\begin{equation} \label{tozero}
\left | \sum_{j=1}^n F(X_{j-1},X_j) \right | \leq c_\epsilon +n \epsilon
\end{equation}
where $c_\epsilon$ is a constant depending on $\epsilon$. As we show in Remark \ref{one} of Chapter 3, in one dimension inequality (\ref{tozero}) follows easily from the ergodic theorem. For dimension greater than one the ergodic theorem is an average over the volume of a rectangle. What we need to prove (\ref{on}) for dimension two or more is an average over paths in the multidimensional integer lattice. Therefore, a direct application of the ergodic theorem will not work for dimension greater than 1. The trick is to use the multivariate ergodic theorem to prove convergence to zero on the fibers of the rectangle and then to use a continuity argument to extend the result to arbitrary paths.

We will use a compactness argument for a certain family of continuous functions which we construct from $f$. Recall that $f(x)=\sum_{0 \rightsquigarrow x} F$. We study a family of functions $\{g_n\}$ which are scaled versions of $f$, that is $g_n(t)=f(nt)/n$. This will allow us to use some results from analysis. In order for $g_n$ to be defined for $t \in \mathbb{R}^d$ we need to extend the domain of $f$ from $\mathbb{Z}^d$ to $\mathbb{R}^d$.  We will define $\hat{f}(t), t \in \mathbb{R}^d$ by interpolating $f$ over $d$-cubes. 

The strategy for the proof is to show that the sequence of functions $\{g_n\}$ converges uniformly to zero on bounded sets. Then for $n$ large enough we will have $|\hat{f}(ns)/n| \leq \epsilon$ which will imply (\ref{on}). The crucial step is to prove that $\{g_n\}$ is equicontinuous and hence compact. To accomplish this fact we rely on a theorem of Garsia, Rodemich and Rumsey (see \cite{stroock:79}) to derive an estimate of the modulus of continuity of the functions $\{g_n\}$ from the moment condtion $\mathbb{E}[|F|^{d+\alpha}]< \infty$.  We begin with the interpolation.

\section{Interpolation}
For clarity we define the interpolation of a function on the cube $[0,1]^d$ and so we need a way to translate an arbitrary point in $\mathbb{R}^d$ to the cube $[0,1]^d$ and back. Let $\pi_t$ be the vector in $\mathbb{Z}^d$ such that for any point $u$ in the cube containing $t$, $u-\pi_t \in [0,1]^d$.
\begin{definition}
For $\tau \in [0,1]^d$, $z \in \mathbb{Z}^d$, and $t \in \mathbb{R}^d$
$$
\tilde{f}(\tau, z) = \sum_{\eta \in \{0,1\}^d} \tau_1^{\eta_1} \tau_2^{\eta_2}\cdots  \tau_d^{\eta_d}(1-\tau_1)^{(1-\eta_1)}\cdots (1-\tau_d)^{(1-\eta_d)} f(z+\eta)
$$
and
$$
\hat{f}(t) = \tilde{f}(t-\pi_t, \pi_t)
$$
We also define,
$$
a_t(\eta,u)=(u-\pi_t)_1^{\eta_1} \cdots (u-\pi_t)_d^{\eta_d}(1-(u-\pi_t)_1)^{1-\eta_1} \cdots (1-(u-\pi_t)_d)^{1-\eta_d} 
$$
so that we can write
$$
\hat{f}(t) = \sum_{\eta \in \{0,1\}^d} a_t(\eta,t) f(\pi_t+\eta)
$$
\begin{remark}
For points on the face of a cube the translation vector will not be unique. This does not create a problem however, since $\hat{f}(t)$ is continuous. The reason for letting $a_t(\eta,u)$ depend on both $t$ and $u$ even though we set $u=t$ above is that if $s$ and $t$ are in the same cube then $\pi_s=\pi_t$ and $a_t(\eta,t)=a_s(\eta,t)$. We will need to use this fact below.
\end{remark}
\end{definition}
We can now define the family of functions $\{g_n\}$
\begin{definition}
$$
g_n(s)= \frac{1}{n} \hat{f}(ns); \quad s \in \mathbb{R}^d
$$
\end{definition}
\section{Ergodic Theorems}
We will make frequent use of Zygmund's multivariate ergodic theorem. We state it here for completeness and derive some useful corollaries. First we need,
\begin{definition}
We say that a function $Y \in  L \log^{d-1} L(\mathbb{P})$ if
$$
\int |Y| \log^{d-1}(|Y| \vee 1) d\mathbb{P} < \infty
$$
where $(a \vee b)=\max(a,b)$.
\end{definition}
\begin{theorem}[Multivaritate Ergodic Theorem, Zygmund] \label{met}
Let $\mpt_1, \ldots  , \mpt_d$ be an ergodic family of $\thinspace \mathbb{P}$-measure preserving transformations that commute. Then for any $Y \in L \log^{d-1} L(\mathbb{P})$, we have
$$
\lim_{n_1,n_2, \ldots, n_d \rightarrow \infty} \frac{1}{n_1 n_1 \cdots n_d}\sum_{i_1=0}^{n_1-1} \cdots \sum_{i_d=0}^{n_d-1} Y(T_1^{i_1} T_2^{i_2} \cdots T_d^{i_d} \omega) = \mathbb{E}[Y] \quad a.s.
$$
\end{theorem}
\begin{proof}
see \cite{Kallenbert:02} pages 186-187.
\end{proof}
The next two corollaries follow immediately.
\begin{corollary} \label{metcor}
Under the assumptions of Theorem \ref{met} we have for $a_i \in (0, \infty)$, $i=1,\ldots, d$.
$$
\lim_{n \rightarrow \infty} \frac{1}{(\lfloor a_1 n \rfloor \lfloor a_2 n \rfloor \cdots \lfloor a_d n \rfloor) }\sum_{i_1=0}^{\lfloor a_1 n\rfloor -1} \cdots \sum_{i_d=0}^{\lfloor a_d n \rfloor -1} Y(T_1^{i_1} T_2^{i_2} \cdots T_d^{i_d} \omega) = \mathbb{E}[Y] \quad a.s.
$$
Here $\lfloor a \rfloor$ is the largest integer not greater than $a$.
\end{corollary}
\begin{corollary}\label{met2}
Let $\mpt_1, \ldots  , \mpt_d$ and their inverses $\mpt_{-1}, \ldots  , \mpt_{-d}$ be an ergodic family of $\mathbb{P}$-measure preserving transformations that commute. Then for any $Y \in L \log^{d-1} L(\mathbb{P})$, we have
$$
\lim_{n \rightarrow \infty} \frac{1}{n^d}\sum_{i_1=-n+1}^{n-1} \cdots \sum_{i_d=-n+1}^{n-1} Y(T_1^{i_1} T_2^{i_2} \cdots T_d^{i_d} \omega) =2^d \mathbb{E}[Y] \quad a.s.
$$
\end{corollary}
\begin{corollary} \label{erg}
Suppose there exist $Y(\omega)$ and $\alpha > 0$ such that $\mathbb{E}[|Y|^{d+\alpha}]<\infty$. Let $\{T_i\}_{i=1}^d$ and their inverses $\{T_{-i}\}_{i=1}^d$ be a $\mathbb{P}$-ergodic family of measure preserving, commuting transformations. Then for $\beta \leq \alpha$ 
$$
\lim_{n \rightarrow \infty} \frac{1}{n^d}\sum_{i_1=-n+1}^{n-1} \cdots \sum_{i_d=-n+1}^{n-1} T_1^{i_1} T_2^{i_2} \cdots T_d^{i_d} |Y(\omega)|^{d+\beta} = 2^d \mathbb{E}[|Y|^{d+\beta}]
$$
\end{corollary}
\begin{proof}
By Corollary \ref{met2} it is sufficient to show
$$
\int_\Omega |Y|^{d+\beta} \log^{d-1}(|Y|^{d+\beta} \vee 1) d\mathbb{P} < \infty
$$
For any $\alpha -\beta>0$, there is a constant $a \geq 1$ such that for $|Y|>a$, $|Y|^{\alpha-\beta} > \log^{d-1} |Y|$. Therefore,
\begin{align*}
\int_\Omega |Y|^{d+\beta} \log^{d-1}(|Y|^{d+\beta} \vee 1)   &\leq \int_{|Y| \leq a} |Y|^{d+\beta} \log^{d-1}(|Y|^{d+\beta} \vee 1)  +\int_{|Y| > a} |Y|^{d+\alpha}\\
& \leq a^{d+\beta} \log^{d-1} (a^{d+\beta}) +\int_{\Omega} |Y|^{d+\alpha} < \infty
\end{align*}
\end{proof}
\section{Equicontinuity}
The core of the proof is to show that $g_n$ is an equicontinous family of functions. We will accomplish this by using the Garsia, Rodemich, Rumsey (GRR) theorem to derive a modulus of continuity from the the integrability condition. The GRR theorem will give us for each $\omega$ and for all $n$ an estimate of the form
$$
|g_n(x)-g_n(y)| \leq c_\omega |y-x|^\delta
$$
where $c_\omega$ is a constant depending on $\omega$ and the dimension $d$, and $\delta>0$. For our purposes the following version of the GRR theorem will suffice. For the proof and the more general version see Stroock and Varadhan \cite{stroock:79}.
\begin{theorem}[Garsia, Rodemich, Rumsey] \label{grr}
Let $h:\mathbb{R}^d \rightarrow \mathbb{R}$ be a continuous function on $B_2$, assume $\gamma > 2d$. If
\begin{equation} \label{grrcond}
\int_{B_1} \int_{B_1} \frac{|h(x)-h(y)|^{d+\alpha}}{|x-y|^\gamma}dx dy \leq c_0
\end{equation}
then for $x,y \in B_1$,
\begin{equation} \label{grrres}
|h(x)-h(y)| \leq c_1 |x-y|^\frac{\gamma-2d}{d+\alpha}
\end{equation}
where $c_1$ depends on $c_0$ and on the dimension $d$.
\end{theorem}
Clearly we need to show that the integral in (\ref{grrcond}) applied to $g_n$ is bounded by a constant independent of $n$ (depending on $\omega$). We divide the domain of integration into two parts giving us the sum of two double integrals. First we integrate over the region $\{x,y\in B_1, |x-y| \leq 2d/n\}$. Then we integrate over $\{x,y \in B_1, |x-y| > 2d/n\}$.
\begin{lemma}
Suppose $F \in \admiss$, then there is a $\beta>0$ such that for $\gamma < 2d+\beta$, the 2-$d$ dimensional integral
$$
\lim_{n \rightarrow \infty} \int_{B_1} \int_{B_{2d/n}(x)\cap B_1} \frac{|g_n(y)-g_n(x)|^{d+\beta}}{|y-x|^{\gamma}} dy dx = c <\infty
$$
\end{lemma}
%-----------------------------------------------------------------------------------------------------------------------------------
\begin{proof}
Choose $\alpha$ so that $F \in L^{d+\alpha}$ and set $\beta = \alpha/2$. By the definition of $g_n$ and the change of variables $t=ny$, $s=nx$, we can write the integral in the lemma as
\begin{align*}
&=\frac{1}{n^{d+\beta}} \int_{B_1} \int_{B_{2d/n}(x)\cap B_1} \frac{|\hat{f}(ny)-\hat{f}(nx)|^{d+\beta}}{|y-x|^{\gamma}} dy dx\\
&=\frac{1}{n^{3d+\beta-\gamma}} \int_{B_n} \int_{B_{2d}(s)\cap B_n} \frac{|\hat{f}(t)-\hat{f}(s)|^{d+\beta}}{|t-s|^{\gamma}} dt ds
\end{align*}
We begin by considering $s$ and $t$ in the same cube. If $s$ and $t$ are in the same cube $\pi_t = \pi_s$ and $\sum_{\eta \in \{0,1\}^d} a_t(\eta,u)=1$. Therefore since $f(\pi_t)$ does not depend on $\eta$,
\begin{align}
|\hat{f}(t)-\hat{f}(s)|&=\left | \sum_{\eta \in \{0,1\}^d} (a_t(\eta,t)-a_t(\eta,s)) (f(\pi_t +\eta)-f(\pi_t))\right| \\
& \leq \sum_{\eta \in \{0,1\}^d}\left| a_t(\eta,t)-a_t(\eta,s)\right| \left |f(\pi_t +\eta)-f(\pi_t) \right |
\end{align}
We now show that  $\left| (a_t(\eta,t)-a_t(\eta,s)\right| \leq c |t-s|$. We use the following lemma,
\begin{lemma} \label{tms}
If $s,t \in [0,1]^n$ then $|t_1 t_2 \cdots t_n - s_1 s_2 \cdots s_n| \leq |t-s|$.
\end{lemma}
\begin{proof}[Proof of Lemma \ref{tms}]
By induction. Clearly the lemma is true for $n=1$. Suppose it is true for $n$, consider,
\begin{align*}
|t_1 t_2 \cdots t_{n+1} &- s_1 s_2 \cdots s_{n+1}| \\
&=|t_1 t_2 \cdots t_{n+1} -t_1 t_2 \cdots t_{n }s_{n+1}+ t_1 t_2 \cdots t_{n}s_{n+1}-s_1 s_2 \cdots s_{n+1}| \\
&=|(t_1 t_2 \cdots t_{n})(t_{n+1}-s_{n+1})+s_{n+1}(t_1 t_2 \cdots t_{n}-s_1 s_2 \cdots s_{n})|\\
\intertext{(and since $|t_1 t_2 \cdots t_n| \leq1$)}
&\leq |t_{n+1}-s_{n+1}|+|t_1 t_2 \cdots t_n - s_1 s_2 \cdots s_n| \\
\intertext{by the induction hypothesis}
&\leq|t_{n+1}-s_{n+1}|+\sum_{i=1}^n|t_i-s_i|=|t-s|
\end{align*}
\end{proof}
We observe that,
$
a_t(\eta,t)-a_t(\eta,s)
$
is the difference between two products, each with at most $d$ factors, so that we may write it as,
$$
t'_1 \cdots t'_d -s'_1 \cdots s'_d  
$$ 
where the $s'_j,t'_j \in [0,1]$. By Lemma \ref{tms} each of these terms is less than or equal to $|t'-s'|=|t-s|$. 
We have shown that for $s$, $t$ in the same cube,
\begin{align*}
|\hat{f}(t)-\hat{f}(s)| &\leq  |t-s| \sum_{\eta \in \{0,1\}^d}|f(\pi_t+\eta)-f(\pi_t)|\\
& \leq  |t-s| \sum_{\eta \in \{0,1\}^d} \sum_{\pi_t \rightsquigarrow \pi_t +\eta} |F|
\end{align*}
If $t$ and $s$ are in different cubes then there is a sequence of points $\{u^{(i)}\}_{i=1}^m$ where $m \leq c$, and $c$ is a constant depending only on the dimension such that the pairs $(s,u^{(1)}), (u^{(i)},u^{(i+1)}), \ldots , (u^{(m)},t)$ are in the same cubes and satisfy for $i=1, \ldots d$,
\begin{equation} \label{eqcond}
\min(s_i,t_i) \leq u^{(1)}_i \leq \ldots \leq u^{(m)}_i \leq \max(s_i,t_i)
\end{equation}
By the triangle inequality,
\begin{align} 
|\hat{f}(t)-\hat{f}(s)| & \leq |\hat{f}(t)-\hat{f}(u^{(m)})| +|\hat{f}(u^{(m)})-\hat{f}(u^{(m-1)})| \\
&+ \cdots +|\hat{f}(u^{(2)})-\hat{f}(u^{(1)})|+|\hat{f}(u^{(1)})-\hat{f}(s)|\\
\intertext{using the above inequality for points in the same cube}
& \leq  \sum_{j=1}^m \sum_{\eta \in \{0,1\}^d} \sum_{\pi_{u^{(j)}} \rightsquigarrow \pi_{u^{(j)}}+\eta} |F| \\
& \cdot \left( |t-u^{(m)}| +|u^{(m)}-u^{(m-1)}| + \cdots +|u^{(2)}-u^{(1)}|+|u^{(1)}-s|\right) \label{tri}
\end{align}
Since $|\cdot|$ is the $\ell_1$-norm, we have for $s_1 \leq u_1 \leq t_1$,
$$
|t_1-s_1| = t_1-s_1=t_1-u_1+u_1-s_1=|t_1-u_1|+|u_1-s_1|
$$
hence, our choice of the $u^{(i)}$, $i=1,\ldots, m$,  guarantees that (\ref{tri}) above is  equal to
$$
|t-s|  \sum_{j=1}^m \sum_{\eta \in \{0,1\}^d} \sum_{\pi_{u^{(j)}}\rightsquigarrow \pi_{u^{(j)}}+\eta} |F| 
$$
Since $|t-s| \leq 2d$ the number of terms in the triple sum is bounded by a constant. Hence it will suffice to show that for any $e \in U$,
$$
\frac{c}{n^{3d+\beta-\gamma}}  \int_{B_n} \int_{B_{2d}(s)\cap B_n} \left | F(\pi_s, \pi_s +e) \right |^{d+\beta} |t-s|^{d+\beta - \gamma} dt ds
$$ 
converges to a finite limit. The previous display is,
\begin{align*}
&\leq \frac{c}{n^{3d+\beta-\gamma}} \int_{B_n}  \left | F(\pi_s, \pi_s +e) \right| ^{d+\beta}   \left (\int_{B_{2d(s)}} |t-s|^{d+\beta - \gamma} dt \right) ds \\
&=\frac{c}{n^{3d+\beta-\gamma}} \int_{B_n}   \left | F(\pi_s, \pi_s +e) \right|^{d+\beta}  ds \int_{B_{2d}} |w|^{d+\beta - \gamma} dw
\end{align*}
The assumption on $\gamma$ guarantees that the righthand integral is finite and that $3d+\beta-\gamma >d$. So to prove the claim we need to show that
$$
\frac{c}{n^{d}} \int_{B_n}   \left | F(\pi_s, \pi_s +e) \right|^{d+\beta}  ds = \frac{c}{n^{d}} \sum_{|u \in \mathbb{Z}^d| \leq n}  \left| F(u,u+e)\right|^{d+\beta} 
$$ converges to a finite limit, but this follows from Corollary \ref{erg}. 
\end{proof}
%-----------------------------------------------------------------------------------------------------------------------------------
We now consider the integral over the region $\{B_1 \cap |x-y| > 2d/n\}$.
\begin{lemma}
Suppose $F \in \admiss$, then there exists $\alpha>0$ so that if $\gamma > 2d+\alpha-1$, then 
\begin{equation} \label{grr2}
\limsup_{n \rightarrow \infty}\int_{B_1} \int_{B^c_{2d/n}(x) \cap B_1} \frac{|g_n(y)-g_n(x)|^{d+\alpha}}{|y-x|^{\gamma}} dy dx  <\infty
\end{equation}
\end{lemma}
%-----------------------------------------------------------------------------------------------------------------------------------
\begin{proof}
Choose $\alpha>0$ so that $F \in L^{d+\alpha}(\mathbb{P})$.
\begin{align*}
 \int_{B_1} & \int_{B^c_{2d/n}(x) \cap B_1}  \frac{|g_n(y)-g_n(x)|^{d+\alpha}}{|y-x|^{\gamma}} dy dx\\
 &= \frac{1}{n^{d+\alpha}}\int_{B_1} \int_{B^c_{2d/n}(x) \cap B_1}  \frac{|\hat{f}(ny)-\hat{f}(nx)|^{d+\alpha}}{|y-x|^{\gamma}} dy dx
\end{align*}
Making the change of variables $t=ny$ and $s=nx$ we have,
$$
\frac{1}{n^{3d+\alpha-\gamma}} \int_{B_n} \int_{B^c_{2d}(s) \cap B_n}  \frac{|\hat{f}(t)-\hat{f}(s)|^{d+\alpha}}{|t-s|^{\gamma}} dt ds
$$
We break up each $d$-dimensional integral into a sum of integrals over $d$-dimensional cubes to arrive at
\begin{equation} \label{cubes}
\leq  \frac{1}{n^{3d+\alpha-\gamma}}  \sum_{\substack{|j-i| > 2d \\ |i| \leq n \\ |j| \leq n}} \int_{i_1}^{i_1 +1} ds_1 \cdots  \int_{i_d}^{i_d +1} ds_d \int_{j_1}^{j_1 +1} dt_1 \cdots  \int_{j_d}^{j_d +1} dt_d \, \frac{|\hat{f}(t)-\hat{f}(s)|^{d+\alpha}}{|t-s|^{\gamma}} 
\end{equation}
For each pair $i,j \in \mathbb{Z}^d$ we are integrating $s$ over one cube and $t$ over another.  Since $|j-i|$ overestimates $|t-s|$ by at most $2d$ we have,
\begin{align*}
|j-i|& \leq |t-s| + 2d \\
& \leq |t-s|+2d|t-s| \\
& \leq 3d|t-s| \\
\frac{1}{3d} |j-i| & \leq |t-s|
\end{align*} 
Hence the left hand side of (\ref{grr2}) is bounded from above by
\begin{equation} \label{mint}
\frac{(3d)^\gamma}{n^{3d+\alpha-\gamma}}  \sum_{\substack{|j-i| > 2d \\ |i| \leq n, |j| \leq n}} \int_{i_1}^{i_1 +1} ds_1 \cdots  \int_{i_d}^{i_d +1} ds_d \int_{j_1}^{j_1 +1} dt_1 \cdots  \int_{j_d}^{j_d +1} dt_d \frac{|\hat{f}(t)-\hat{f}(s)|^{d+\alpha}}{|j-i|^{\gamma}} 
\end{equation}
We write 
\begin{align*}
\hat{f}(t)-\hat{f}(s)&=f(\pi_t)-f(\pi_s)\\
&+\sum_{\eta \in \{0,1\}^d} a_t(\eta,t)(f(\pi_t+\eta)-f(\pi_t))\\
&-\sum_{\eta \in \{0,1\}^d} a_t(\eta,s)(f(\pi_s+\eta)-f(\pi_s))
\end{align*}
Since $0 \leq a_t \leq 1$, Holder's inequality gives
\begin{align*}
|\hat{f}(t)-\hat{f}(s)|^{d+\alpha} &\leq  c_0 |\pi_t-\pi_s|^{d+\alpha-1}\sum_{\pi_s \rightsquigarrow \pi_t} |F|^{d+\alpha}\\
&+c_0 \sum_{\eta \in \{0,1\}^d} \left |(f(\pi_t+\eta)-f(\pi_t) \right|^{d+\alpha} \\
&+c_0 \sum_{\eta \in \{0,1\}^d} \left | (f(\pi_s +\eta)-f(\pi_s) \right| |^{d+\alpha}
\end{align*}
Since the integrand in (\ref{mint}) is symmetric in $s$ and $t$  and constant on each $d$-cube we  obtain after another application of Holder's inequality,
\begin{align*}
&\frac{1}{n^{3d+\alpha-\gamma}}  \iint_{\substack{|t-s| > 2d \\ |s| \leq n, |t| \leq n}} \frac{|\hat{f}(t)-\hat{f}(s)|^{d+\alpha}}{|t-s|^{\gamma}} dt ds \\
&\leq \frac{c_1}{n^{3d+\alpha-\gamma}}  \sum_{\substack{i,j \in [-n,n]^d \cap \mathbb{Z}^d \\i \neq j}} \left[ |j-i|^{d+\alpha-1}\frac{\sum_{i \rightsquigarrow j} |F|^{d+\alpha}}{|j-i|^\gamma} + \frac{2\sum_{\eta \in \{0,1\}^d} \sum_{j \rightsquigarrow j+\eta}|F|^{d+\alpha}}{|j-i|^\gamma} \right ]
\end{align*}
Consider the second term in the sum
\begin{align} \label{grrs2}
\frac{c_1}{n^{3d+\alpha-\gamma}}  &\sum_{\substack{i,j \in [-n,n]^d \cap \mathbb{Z}^d \\i \neq j}}   \frac{2\sum_{\eta \in \{0,1\}^d} \sum_{j \rightsquigarrow j+\eta}|F|^{d+\alpha}}{|j-i|^\gamma}\\
&  \leq \frac{c_1}{n^{3d+\alpha-\gamma}}  \sum_{\substack{i,j \in [-n,n]^d \cap \mathbb{Z}^d \\i \neq j}} \sum_{e \in U} \frac{|F(j,j+e)|^{d+\alpha}}{|j-i|^\gamma}
\end{align}
Consider the coefficient of an arbitrary term in the sum (\ref{grrs2}) of $|F(j,j+e_1)|^{d+\alpha}$.
Which is certainly smaller than
\begin{align*}
\frac{c_1}{n^{3d+\alpha-\gamma}}& \sum_{i \in \mathbb{Z}^d \setminus 0} \frac{1}{|j-i|^\gamma} \\
&=\frac{c}{n^{3d+\alpha-\gamma}}\sum_{i \in \mathbb{Z}^d \setminus 0} \frac{1}{|i|^\gamma} \\
&\leq \frac{c_2}{n^{3d+\alpha-\gamma}}
\end{align*}
Provide $\gamma > 1$. Therefore, as $n$ tends to $\infty$ the right hand side of inequality (\ref{grrs2}) will converge if for any $e \in U$, the following sum converges
$$
\frac{c_2}{n^{3d+\alpha-\gamma}} \sum_{j \in \mathbb{Z}^d} |F(j,j+e)|^{d+\alpha}
$$ 
which converges by Corollary \ref{erg} for $\gamma \leq 2d + \alpha$.
As for the first term,
$$
\frac{c_0}{n^{3d+\alpha-\gamma}} \sum_{\substack{i,j \in [-n,n]^d \cap \mathbb{Z}^d \\i \neq j}} \frac{\sum_{i\rightsquigarrow j} |F|^{d+\alpha}}{|j-i|^{\gamma-(d-1)-\alpha}}
$$
Again  consider the coefficient of the term $|F(k,k+e_1)|^{d+\alpha}$. This term is only included in sums where $i'=k'$ with $i=(i_0,i')$. To see this observe that since the canonical path from $i$ to $j$ starts with a move in the $e_1$ direction, a path starting at $i$ will not cross $(k,k+e_1)$ unless $k'=i'$. Thus the coefficient is
\begin{align*}
\frac{c_0}{n^{3d+\alpha-\gamma}} &\sum_{i_0} \sum_{j_0} \sum_{j'} \frac{1}{|j-i|^{\gamma-(d-1)-\alpha}} \\
&\leq
\frac{c_1}{n^{3d+\alpha-\gamma}} \sum_{i_0} \sum_{j_0} \sum_{j'} \frac{1}{|j_0-i_0|^{\gamma-(d-1)-\alpha}+|j'-i'|^{\gamma-(d-1)-\alpha}}
\end{align*}
Where the sums here and in what follows are taken over $i \neq j, |i|,|j| \leq n$ Let $a:= |j_0 - i_0|^{\gamma-(d-1)-\alpha}$, we need to calculate
$$
\frac{c_1}{n^{3d+\alpha-\gamma}} \sum_{j'} \frac{1}{a+|j'-i'|^{\gamma-(d-1)-\alpha}}=\frac{c_1}{n^{3d+\alpha-\gamma}} \sum_{j'} \frac{1}{a+|j'|^{\gamma-(d-1)-\alpha}}
$$
This sum converges iff $\gamma - (d-1) - \alpha > d-1$, that is if $\gamma > 2d+ \alpha -2$ which is assumed to be true. Indeed a straight forward calculus computation shows
$$
\sum_{j'} \frac{1}{a+|j'|^{\gamma-(d-1)-\alpha}} \leq c_2 a^{\frac{d}{\gamma -(d-1)-\alpha}-1}
$$
Substituting in the value of $a$ we have
$$
\sum_{j'} \frac{1}{a+|j'|^{\gamma-(d-1)-\alpha}} \leq c_2 |j_0-i_0|^{2d-\gamma +\alpha -1}
$$
Hence,
\begin{align*}
\frac{c_0}{n^{3d+\alpha-\gamma}} \sum_{i_0} \sum_{j_0} \sum_{j'} \frac{1}{|j-i|^{\gamma-(d-1)-\alpha}} &\leq \frac{c_2}{n^{3d+\alpha-\gamma}} \sum_{i_0} \sum_{j_0}  |j_0-i_0|^{2d-\gamma +\alpha -1}\\
& \leq \frac{c_3}{n^{3d+\alpha-\gamma}} n^{2d-\gamma -\alpha}\\
& \leq \frac {c_3}{n^d}
\end{align*}
Finally,
$$
\frac{c_0}{n^{3d+\alpha-\gamma}} \sum_{i} \sum_{j} \frac{|f(j)-f(i)|^{d+\alpha}}{|j-i|^\gamma} \leq \frac {c_3}{n^d} \sum_{e \in U} \sum_{k \in \mathbb{Z}^d} |F(k,k+e|^{d+\alpha}
$$
which converges again by an application of Corollary \ref{erg}.
\end{proof}
%-----------------------------------------------------------------------------------------------------------------------------------
\begin{lemma}\label{equi}
The family of functions $\{g_n\}$ is equicontinous.
\end{lemma}
%-----------------------------------------------------------------------------------------------------------------------------------
\begin{proof}
The previous two lemmas imply that there is an $\alpha >0$ such that with $\beta=\alpha/2$,
$$
\sup_n \iint_{|y-x| \leq 1} \frac{|g_n(y)-g_n(x)|^{d+\beta}}{|y-x|^{\gamma}} dy dx < B
$$ for $\gamma < 2d + \beta$, $\beta=\alpha/2$, and $B$ a constant. By theorem \ref{grr}
\begin{align*}
|g_n(y)-g_n(x)|& \leq  c |x-y|^{\frac{\gamma-2d}{d+\beta}}
\end{align*} 
 Provided $\gamma > 2d$ and $x,y$ are in the unit $L_1$ ball. The lemma is proved by choosing $2d < \gamma < 2d + \alpha/2$.
\end{proof}
\section{Proof of Theorem \ref{fon}}
\begin{lemma} \label{intab}
The sequence $\{g_n\}$ has a subsequence that converges uniformly on compacts to a function $g$. (Indeed any subsequence has a further subsequence which is convergent). We denote the subsequence by $\{g_n\}$ as well.  Additionally, for all $b,a \in \mathbb{R}^d$,
$$
\int_{a_d}^{b_d} \cdots \int_{a_2}^{b_2} g(b_1,y_2, \ldots, y_d)-g(a_1,y_2, \ldots, y_d)dy_2 \cdots dy_d =0
$$
\end{lemma}
\begin{proof}
The existence of the convergent subsequence follows from Lemma \ref{equi}. For the second assertion, since any $d$-dimensional rectangle can by created by adding and subtracting rectangles with a  corner at the origin it will suffice to prove,
$$
\int_0^{a_d} \cdots \int_0^{a_2} g(a_1,y_2, \ldots, y_d)-g(0,y_2, \ldots, y_d)dy_2 \cdots dy_d =0
$$
Consider the sum,
\begin{equation*}
\frac{1}{(\lfloor a_1 n \rfloor \lfloor a_2 n \rfloor \cdots \lfloor a_d n \rfloor) } \sum_{i_d =0}^{\lfloor a_d n \rfloor -1} \cdots \sum_{i_1 =0}^{\lfloor a_1 n \rfloor -1} T_{e_1}^{i_1} \cdots T_{e_d}^{i_d} F(\omega,e_1)
\end{equation*}
Since $\mathbb{E}[F(\omega,e_1)]=0$ this average tends to 0 as $n \rightarrow \infty$ by Corollary \ref{metcor}  provided that $F$ is in $L \log^{d-1} L (\mathbb{P})$, which is true since we are assuming that $\mathbb{E}[|F|^{d+\alpha}] < \infty$. It follows immediately that
\begin{equation} \label{avg}
\lim_{n \rightarrow \infty} \frac{1}{n^d} \sum_{i_d =0}^{\lfloor a_d n \rfloor -1} \cdots \sum_{i_1 =0}^{\lfloor a_1 n \rfloor -1} T_{e_1}^{i_1} \cdots T_{e_d}^{i_d} F(\omega,e_1)=0
\end{equation}
On the other hand, changing notation shows that 
$$
\frac{1}{n^d} \sum_{i_d =0}^{\lfloor a_d n \rfloor -1} \cdots \sum_{i_1 =0}^{\lfloor a_1 n \rfloor -1} T_{e_1}^{i_1} \cdots T_{e_d}^{i_d} F(\omega,e_1)
$$
is equivalent to
\begin{align*}
\frac{1}{n^d}& \sum_{i_d =0}^{\lfloor a_d n \rfloor -1} \cdots \sum_{i_2 =0}^{\lfloor a_2 n \rfloor -1}f((\lfloor a_1 n\rfloor -1)e_1+i_2 e_2+\cdots +i_d e_d)-f(i_2 e_2+\cdots +i_d e_d)\\
&=\frac{1}{n^{d-1}} \sum_{i_d =0}^{\lfloor a_d n \rfloor -1} \cdots \sum_{i_2 =0}^{\lfloor a_2 n \rfloor -1} \Big [g_n(e_1(\lfloor a_1 n\rfloor -1)/n +e_2 i_2/n +\cdots +e_d i_d/n)\\
&-g_n(e_2 i_2/n +\cdots +e_d i_d/n)\Big ]
\intertext{by the definintion of $g_n$.}
\end{align*}
We want to show that the above converges to the desired integral so we consider,
\begin{align*}
&\left | \frac{1}{n^{d-1}} \sum (g_n-g) +\frac{1}{n^{d-1}} \sum g - \int g \right | \\
&\leq \left | \frac{1}{n^{d-1}} \sum (g_n-g) \right |+\left |\frac{1}{n^{d-1}} \sum g - \int g \right | 
\end{align*}
The second term converges to zero since $\frac{1}{n^{d-1}} \sum g_n$ is a Reimann sum. As $n \rightarrow \infty$ the first summand tends to zero since $|g_n-g|$ converges uniformly to zero.
\end{proof}
We have assembled all of the ingredients needed to prove the main theorem.
\begin{proof}[Proof of Theorem \ref{fon}]
The uniform limit of a sequence of continuous functions $g$ is continuous. Since $a$ and $b$ are abitrary Lemma \ref{intab} shows that the function $g(x,y_2, \ldots, y_d)$ is a constant function of $x$. By analogous reasoning we see that $g$ is in fact constant in every coordinate, that is, $g$ is a constant function. Since $g(0)$ is 0, $g$ must be identically 0. Hence all convergent subsequences converge to 0 and therefore $g_n $ converges uniformly to 0 on the unit $d$-dimensional cube. In other words, for any $\epsilon >0$ and $n$ large enough we have $|g_n(s)| \leq \epsilon$ for $s \in [0,1]^d$. So that $|\hat{f}(ns)/n|\leq \epsilon$. To conclude the proof we need to show that for any $\epsilon >0$ there is an $N$ such that for $n \geq N$,
$$
\sup_{\substack{|z| \leq n \\ z \in \mathbb{Z}^d}} \left |\frac{f(z)}{n} \right |\leq \epsilon
$$
Let $z$ be arbitrary such that $|z| \leq n$, then there is a vector $s \in [0,1]^d$ such that $z=ns$. By what we have shown so far we now have
$|\hat{f}(z)/n|\leq \epsilon$ and since $z \in \mathbb{Z}^d$ we can write $|f(z)/n|\leq \epsilon$. Since $z$ is arbitrary the proof is concluded.
\end{proof}
 
\chapter{Logarithmic Moment Generating Function}
In this chapter we prove Theorem \ref{a}. Recall the definition of the logarithmic moment generating function:
$$
\Lambda(\lambda) = \inf_{F \in \admiss} \esup_\omega \log \sum_{e \in U} p(\omega ,e) \base^{\lb \lambda , e \rb+F(\omega,e)}
$$
with the $\esup$ taken with respect to the measure $\mathbb{P}$. Theorem \ref{a} asserts that $\Lambda(\lambda)$ is the logarithmic moment generating function - i.e., the limit as $n$ tends to $\infty$ of $\frac{1}{n} \log E^{P_\omega}[\base^{\lb \lambda, Xn\rb} ]$. To prove this assertation we first derive a lower bound $\Gamma(\lambda)$ for $\frac{1}{n} \log E^{P_\omega}[\base^{\lb \lambda, Xn\rb} ]$  using standard methods from the theory of large deviations. Then we derive the upper bound which turns out to be $\Lambda(\lambda)$; this is where we make crucial use of Theorem \ref{fon}. Finally we show that $\Lambda(\lambda) \leq \Gamma(\lambda)$ by establishing the existence of a family of functions $\{F_\epsilon\} \in \admiss$ so that for any $\epsilon > 0$
$$
\esup_\omega \log \sum_{e \in U} p(\omega,e) \base^{\lb \lambda, e \rb +F_{\epsilon} (\omega,e)} \leq \Gamma(\lambda)+\epsilon
$$
Then since,
\begin{align*}
 \inf_{F \in \admiss} \esup_\omega \log \sum_{e \in U} p(\omega,e) \base^{\lb \lambda, e \rb +F (\omega,e)} & \leq \inf_\epsilon \esup_\omega \log \sum_{e \in U} p(\omega,e) \base^{\lb \lambda, e \rb +F_{\epsilon} (\omega,e)}
 \intertext{and}
 \inf_\epsilon \Gamma(\lambda) + \epsilon &= \Gamma(\lambda)
\end{align*}
we will have our result.
\section{Lower Bound}
The lower bound is established in a straightforward manner. After changing measure we take the supremum over all pairs of transition functions $q$ and densities $\phi$ where $\phi$ is an ergodic, invariant density for the Markov chain $q$. In order that we can take this supremum over arbitrary pairs $(q,\phi)$ we introduce a function $h$ into the objective function which will force the objective funtion to be negative infinity if $\phi$ is not an ergodic, invariant density for $q$. The result is a lower bound of $\Gamma(\lambda)$ which we define presently.
\begin{definition}
$\Gamma(\lambda)$ is defined as:
\begin{equation}
 \sup_{(q,\phi)} \inf_h  \int   \left(\sum_{e \in U} \lb \lambda ,e \rb  - \log \frac{q(\omega, e)}{p(\omega,e)}+h(\omega)-h(\mpt_e \omega)\right)q(\omega,e) \phi(\omega) d \mathbb{P}
\end{equation}
where  $q$ is a transition function, $\phi$ a probability density and $h$ a bounded measurable function.
\end{definition}
We obtain the lower bound by a standard change of measure argument.  We write $E^{P_\omega}\left[ \base^{\lb \lambda, X_n\rb}\right]=E^{Q_\omega}\left[ \base^{\lb \lambda, X_n \rb}\frac{dP_\omega}{dQ_\omega}\right]$ for a Markov chain $Q_\omega$.  Then assuming that the  measure $Q_\omega$ is absolutely continuous with respect to $P_\omega$ and is a stationary, ergodic Markov chain on the space of environments we use first the ergodic theorem and then the law of large numbers for RWRE to take the limit as $n$ tends to $\infty$. Using Jensen's inequality and taking the supremum over all such measures  $Q_\omega$ will yield a lower bound. In order to prove that $\Lambda(\lambda) \leq \Gamma(\lambda) $ below we will need to replace this expression for the lower bound with one where the supremum is taken over all pairs $q$ and $\phi$ (not only those where $\phi$ is the ergodic, invariant density for the Markov chain with transition function $q$). To do this we incorporate the condition for a stationary ergodic density into the expression for the lower bound. We first state the law of large numbers for the RWRE. The proof is a straightforward generalization of that given by Sznitman in \cite{sznitman:02} pages 14-15.
\begin{theorem} \label{slln}
Suppose $X_n$ is a RWRE with transition function $p$ and probability measure on the environmnet $\mathbb{P}$. If  $\mathbb{P}'$ is an invariant probability measure for the Markov chain then $P_\omega$-a.s.
$$
\lim_{n \rightarrow \infty} \frac{1}{n} \lb \lambda, X_n \rb =  \int \sum_{e \in U} \lb \lambda,e \rb p(\omega,e)  d\mathbb{P}'(\omega)
$$
\end{theorem}
The lower bound is proved as the following theorem,
\begin{theorem} \label{lowbdd}
\begin{equation}
\liminf_{n \rightarrow \infty} \frac{1}{n} \log E^{P_\omega}\left[ \base^{\lb \lambda, X_n \rb}\right] \geq \Gamma(\lambda)
\end{equation}
\end{theorem}
\begin{proof}
For the change of measure we  use the explicit formula for the Radon Nikodym derivative of one Markov chain with respect to another, indeed
\begin{align*}
E^{P_\omega}\left[ \base^{\lb \lambda, X_n\rb}\right]&=E^{Q_\omega}\left[ \base^{\lb \lambda, X_n \rb}\frac{dP_\omega}{dQ_\omega}\right]\\
&=E^{Q_\omega}\left[\exp \left\{\lb \lambda, X_n \rb-\log \frac{q(X_0,X_1)\cdots q(X_{n-1},X_{n})}{p(X_0,X_1)\cdots p(X_{n-1},X_{n})}\right\}\right] \\
&=E^{Q_\omega}\left[\exp \left\{\lb \lambda, X_n \rb-\sum_{k=0}^{n-1}\log \frac{q(X_k,X_{k+1})}{p(X_k,X_{k+1})}\right\}\right] 
\end{align*} 
where $Q_\omega$ is the Markov chain with transition function $q$ and intial state $\omega$. By Jensen's inequaltity 
$$
\liminf_{n \rightarrow \infty} \frac{1}{n} \log E^{P_\omega}\left[ \base^{\lb \lambda, X_n \rb}\right] \geq \liminf_{n \rightarrow \infty} E^{Q_\omega}\left[\frac{1}{n} \lb \lambda, X_n \rb-\frac{1}{n}\sum_{k=0}^{n-1}\log \frac{q(X_k,X_{k+1})}{p(X_k,X_{k+1})}\right] 
$$
%-----------------------------------------------------------------------------------------------------------------------------------
First we consider
% \frac{1}{n}E^{Q_\omega}\left[\log\frac{dQ_\omega}{dP_\omega}\right] 
$$
 E^{Q_\omega}\left[\frac{1}{n}\sum_{k=0}^{n-1}\log \frac{q(X_k,X_{k+1})}{p(X_k,X_{k+1})}\right] 
 $$
By definition of $p$ and $q$ this expression can be written as,
$$ 
E^{Q_\omega}\left[\frac{1}{n}\sum_{k=0}^{n-1}\log \frac{q(\overline{\omega}_k,\overline{\omega}_{k+1})}{p(\overline{\omega}_k,\overline{\omega}_{k+1})}\right] 
$$
We restrict $\mathbb{Q}$ to be absolutely continuous with respect to $\mathbb{P}$ so that we have $d\mathbb{Q}=\phi(\omega)d\mathbb{P}$. We also take $\phi d \mathbb{P}$ to be an ergodic invariant measure for the Markov chain  $Q_\omega$, then by the tower property of conditional expectation and the ergodic theorem we have
\begin{align}\label{dqdp}
\liminf_{n \rightarrow \infty}E^{Q_\omega} & \left[\frac{1}{n}  \sum_{k=0}^{n-1}\log \frac{q(X_k,X_{k+1})}{p(X_k,X_{k+1})}\right] \\
&=\liminf_{n \rightarrow \infty} E^{Q_\omega} \left[\frac{1}{n}\sum_{k=0}^{n-1}  \sum_{e \in U} \log \frac{q(\overline{\omega}_k,\mpt_e \overline{\omega}_{k})}{p(\overline{\omega}_k,\mpt_e \overline{\omega}_{k})}q(\overline{\omega}_k,\mpt_e \overline{\omega}_{k})\right] \\
&=\int \sum_{e \in U} \log \frac{q(\omega,e)}{p(\omega,e)}q(\omega,e) \phi(\omega) d \mathbb{P}(\omega) \label{dqdp}
\end{align}
To obtain a lower bound it remains to evaluate
$$
\lim_{n \rightarrow \infty} E^{Q_\omega}\left[\frac{1}{n} \lb \lambda, X_n\rb \right]
$$
We let $A$ be the set of all pairs $(q,\phi)$ where $\phi$ is the ergodic, invariant density for the Markov chain with transition function $q$. For $(q,\phi)\in A$ we apply Theorem \ref{slln}, the strong law of large numbers. 
\begin{align} \label{sl}
\lim_{n \rightarrow \infty} E^{Q_\omega}\left[\frac{1}{n} \lb \lambda, X_n\rb \right]&= \int \sum_{e \in U} \lb \lambda,e \rb q(\omega,e)  \phi(\omega)d\mathbb{P}(\omega)
\end{align}
Therefore combining the results  (\ref{dqdp}) and (\ref{sl}) we get
\begin{align} \label{lb1} 
\liminf_{n \rightarrow \infty} \frac{1}{n} \log E^{P_\omega}\left[ \base^{\lb \lambda, X_n \rb}\right] 
& \geq   \sup_{(q,\phi)\in A}  \int     \sum_{e \in U} \lb \lambda,e \rb  q(\omega,e)- \log \frac{q(\omega,e)}{p(\omega,e)}q(\omega,e)\phi d\mathbb{P}
\end{align}
The final step of the proof is to remove the restriction on $\phi$ and $q$ as explained above. The condition for $\phi(\omega)$ to be an invariant density for the chain is
\begin{align} \label{stat}
\int h(\omega) \phi d\mathbb{P}
&=\int \sum_{e \in U} h(\mpt_e \omega) q( \omega,e)\phi d\mathbb{P}
\end{align} 
for all bounded measurable functions $h$. It turns out that this condition guarantees ergodicity as well (see, e.g.,~\cite{sznitman:02}). Therefore if $(q,\phi) \notin A$ then 
\begin{equation} \label{idcon}
\inf_h \int \sum_{e \in U} (h(\omega)-h(\mpt_e \omega)) q( \omega,e)\phi d\mathbb{P}= -\infty
\end{equation}
This allows us to use (\ref{stat})  to decouple  $\phi$ and $q$ in (\ref{lb1}) and to take the supremum over all $\phi$ and $q$. The result is,
\begin{align*}
& \sup_{(q,\phi)\in A}  \int     \sum_{e \in U} \lb \lambda,e \rb  q(\omega,e)- \log \frac{q(\omega,e)}{p(\omega,e)}q(\omega,e)\phi d\mathbb{P}\\
&=\sup_{(q,\phi)} \inf_h  \int   \left(\sum_{e \in U} \lb \lambda ,e \rb  - \log \frac{q(\omega, e)}{p(\omega,e)}+h(\omega)-h(\mpt_e \omega)\right)q(\omega,e) \phi(\omega) d \mathbb{P} \\
\end{align*}
which is the expression for $\Gamma(\lambda)$ concluding the proof.
\end{proof}
%-----------------------------------------------------------------------------------------------------------------------------------
%-----------------------------------------------------------------------------------------------------------------------------------
%-----------------------------------------------------------------------------------------------------------------------------------
\section{Upper Bound}
To derive the upper bound for $1/n \log E^{P_\omega}\left [\base^{\lb \lambda, X_n \rb}\right ] $ we would like to bound the increments $E^{P_\omega}\left [\base^{\lb \lambda, X_n-X_{n-1} \rb} | X_{n-1}\right ] $ by  a constant, say $ \base^c$. Then using the tower property of conditional expectation to  iterate this inequality we could show that  $E^{P_\omega}\left [\base^{\lb \lambda, X_n \rb}\right ]$ is bounded above by $\base^{cn}$, and after taking the $\log$ and dividing by $n$ we would have an estimate for $1/n \log E^{P_\omega}\left [\base^{\lb \lambda, X_n \rb}\right ]$ of $c$. This program will not work quite so simply because we do not know how to arrive at the value of $c$; in fact since $c$ is the quantity we are trying to derive in the first place, we seem to have made no progress. However,  we can easily  obtain an upper bound for $E^{P_\omega} \left [\base^{\lb \lambda,X_n-X_{n-1}\rb+F(X_{n-1},X_n)}\Big |X_{n-1}\right]$ with the appropriate choice of $F$. Proposition \ref{fon} shows that $1/n \sum_{i=1}^n F(X_{n-1},X_n)$ is small (in the appropriate sense), allowing us to derive an upper bound. Written as a variational formula.  
\begin{theorem} \label{ub}
\begin{equation}
\limsup_{n \rightarrow \infty} \frac{1}{n} \log E^{P_\omega}\left [\base^{\lb \lambda, X_n \rb}\right ] \leq \Lambda(\lambda)
\end{equation}
\end{theorem}
%-----------------------------------------------------------------------------------------------------------------------------------
\begin{definition} \label{defK}
For a function $F(\omega,e)$ we set 
$$
K(F):= \esup_\omega \log \sum_{e \in U} p(\omega,e) \base^{\lb \lambda,e \rb+F( \omega,e)}
$$
\end{definition}
Definition \ref{defK} is useful because for all functions $F(\omega,e)$ we have the obvious inequality
\begin{equation} \label{obv}
\sum_{e \in U} p(\omega,e) \base^{\lb \lambda,e \rb+F( \omega,e)} \leq \base^{K(F)}
\end{equation}
Using  inequality (\ref{obv}) gives,
\begin{align} \label{cond} 
E^{P_\omega} \left [\base^{\lb \lambda,X_n-X_{n-1}\rb+F(X_{n-1},X_n)}\Big |X_{n-1}\right]=\sum_{e \in U} p(\overline{\omega}_{n-1},e) \base^{\lb \lambda,e \rb+F( \overline{\omega}_{n-1},e)} \leq \base^{K(F)}
\end{align}
We can now perform the iteration mentioned above to obtain the dersired upper bound for $\exp\left\{\lb \lambda, X_n \rb+\sum_{j=1}^n F(X_{j-1},X_j)\right\}$. This is done in the following lemma.
\begin{lemma} \label{smart}
$S_n := \exp\left\{\lb \lambda, X_n \rb+\sum_{j=1}^n F(X_{j-1},X_j)-nK(F)\right\}$ is a supermartingale with respect to the sigma field $\sigma(X_0,X_1, \ldots, X_n)$.
\end{lemma}
\begin{proof}
Since $X_n$ is Markov, using inequality (\ref{cond}) gives
\begin{align*}
E^{P_\omega} &\bigl[S_n \bigr | \sigma(X_0, \ldots, X_{n-1})] = E^{P_\omega}\bigl[S_n \bigr |  X_{n-1}]\\
&=E^{P_\omega}\left[\exp\Bigl\{\lb \lambda, X_n \rb+\sum_{j=1}^n F(X_{j-1},X_j)-nK\Bigr\} \Big | X_{n-1}\right]\\
&= E^{P_\omega}\left[S_{n-1}\exp\bigl\{ \lb \lambda, X_n-X_{n-1} \rb+F(X_{n-1},X_n)-K\bigr\} \Big | X_{n-1}\right]\\
&= \base^{-K} S_{n-1} E\left[\exp\left\{\lb \lambda, X_n-X_{n-1} \rb+F(X_{n-1},X_n)\right\} \Big | X_{n-1}\right]  \\
&\leq S_{n-1} \quad \text{by inequality (\ref{cond})).}
\end{align*}
\end{proof}
%-----------------------------------------------------------------------------------------------------------------------------------
\begin{proof}[Proof of Theorem \ref{ub}]
Applying Lemma \ref{smart} gives,
\begin{align}\label{mart}
E^{P_\omega}\left[\exp\left\{\lb \lambda, X_n \rb+\sum_{j=1}^n F(X_{j-1},X_j)\right\}\right]&\leq \base^{nK(F)}
\end{align}
We now restrict our attention to functions $F$ in class $\admiss$ and use Theorem \ref{fon}  to show that for all $F \in \admiss$,  given $\epsilon >0$, there is a constant $c_\epsilon \geq 0$ such that
\begin{equation} \label{son}
\sum_{j=1}^n F(X_{j-1},X_j) \geq -c_\epsilon - n \epsilon
\end{equation}
According to  Theorem \ref{fon}, for $F \in \admiss$,
$$\lim_{n \rightarrow \infty} \sup_{\substack{|z| \leq n \\ z \in \mathbb{Z}^d}} \frac{f(z)}{n}=0$$
That is that given $\epsilon >0$, there exists $N_\epsilon$ such that for $n \geq N_\epsilon$,
$$
\Big |\sup_{\substack{|z| \leq n \\ z \in \mathbb{Z}^d}} f(z)\Big| \leq n \epsilon
$$
On the other hand, for $n < N_\epsilon$
$$
\left |\sup_{|z| < n} f(z)\right| \leq \left |\sup_{|z| < N_\epsilon} f(z)\right| \leq c_\epsilon
$$
Hence, for any $n$,
$$
\left|\sup_{|z| < n} f(z)\right| \leq n \epsilon + c_\epsilon
$$
which is exactly what we need to verify that
$$
\left |\sum_{j=1}^n F(X_{j-1},X_j) \right | \leq c_\epsilon + n \epsilon
$$
Hence using (\ref{son})
\begin{align*}
E^{P_\omega}\left[\exp\left\{\lb \lambda, X_n \rb  -c_\epsilon -n \epsilon \right\}\right]  &\leq E^{P_\omega}\left[\exp\left\{\lb \lambda, X_n \rb+\sum_{j=1}^n F(X_{j-1},X_j)\right\}\right]\\
\intertext{and}
E^{P_\omega}\left[\base^{\lb \lambda, X_n \rb}\right]  &\leq \base^{n(K(F)+\epsilon) +c_\epsilon}
\intertext{so that}
\frac{1}{n} \log E^{P_\omega}\left [\base^{\lb \lambda, X_n \rb} \right] & \leq K(F)+\epsilon+c_\epsilon/n\\
\intertext{Letting $n \rightarrow \infty$ then taking  the $\inf$ over $F \in \admiss$ gives us the dersired upper bound,}
\limsup_{n \rightarrow \infty} \frac{1}{n} \log E^{P_\omega}\left [\base^{\lb \lambda, X_n \rb} \right] & \leq \inf _{F \in \admiss} K(F)\\
&= \inf _{F \in \admiss} \esup_\omega \log \sum_{e \in U} p(\omega,e) \base^{\lb \lambda,e \rb+F( \omega,e)}
\end{align*}
\end{proof}
\begin{remark} \label{one}
In one dimension (\ref{son}) follows readily from the ergodic theorem. Indeed, 
\begin{align*}
\frac{1}{n}\sum_{j=1}^n F(j-1,j)&=\frac{1}{n}\sum_{j=1}^n F(\mpt_1^{j-1} \omega,e_1)\\
\intertext{and}
\frac{1}{n}\sum_{j=1}^n F(1-j,-j)&=\frac{1}{n}\sum_{j=1}^n F(\mpt_1^{1-j} \omega,-e_1)\\
&=\frac{1}{n}\sum_{j=1}^n F(\mpt_{-1}^{j-1} \omega,-e_1)
\end{align*}
and the righthand sides both converge to zero by the ergodic theorem. This implies that for each $\epsilon > 0$ there exists $c_\epsilon \geq 0$ such that for all $n$,
$$
\left |\sum_{j=1}^n F(X_{j-1},X_j)\right| \leq c_\epsilon + n\epsilon
$$
\end{remark}
%-----------------------------------------------------------------------------------------------------------------------------------
%-----------------------------------------------------------------------------------------------------------------------------------
%-----------------------------------------------------------------------------------------------------------------------------------
\section{Equivalence of Upper and Lower Bounds.}
At this point we have proven in Theorems \ref{lowbdd} and \ref{ub} that
$$
\Gamma(\lambda) \leq \liminf_{n \rightarrow \infty} \frac{1}{n} \log E^{P_\omega}\left [\base^{\lb \lambda, X_n \rb}\right ] \leq \limsup_{n \rightarrow \infty} \frac{1}{n} \log E^{P_\omega}\left [\base^{\lb \lambda, X_n \rb}\right ]  \leq \Lambda(\lambda)
$$
Therefore the proof of Theorem \ref{a} will be complete once we show that $\Lambda(\lambda) \leq \Gamma(\lambda)$.  We demonstrated at the beginning of this chapter that it will suffice to prove
\begin{theorem} \label{uls}
For each $\epsilon > 0$ there exists a function $F_\epsilon \in \admiss$ such that
$$
\esup_\omega \log \sum_{e \in U} p(\omega, e) \base^{\lb \lambda, e \rb + F_\epsilon (\omega,e)}  \leq \Gamma(\lambda)+\epsilon
 $$
\end{theorem}
\begin{proof}
We would like to exchange the order of the $\inf$ and $\sup$ in $\Gamma(\lambda)$, but we cannot apply the minimax theorem here because we cannot find a topology that simultaneously make the spaces of $q$ and $\phi$ compact and the objective function upper semicontinuous. Instead we take the supremum over a compact space to construct a sequence $\{F_{k,\epsilon} \}$ that converges weakly to $F_\epsilon$.

We assume the $\mathcal{F}=\sigma(\cup_k \mathcal{D}_k)$ where each $\mathcal{D}_k$ is a finite $\sigma$-algebra. Let $\mathcal{E}_1=\mathcal{D}_1$ and for $k>1$ set 
$$
\mathcal{E}_k=\sigma\left (\left ( \bigcup_{e \in U} \mpt_e \mathcal{E}_{k-1}\right) \bigcup \mathcal{D}_k\right)
$$
so that $\mpt_e \mathcal{E}_k \supset \mathcal{E}_{k-1}$ and $\mpt_{-e} \mathcal{E}_k \supset \mathcal{E}_{k-1}$.

Let $A_k$ be the set of $\mathcal{E}_k$ measurable probability densities, that is functions $\phi$ which are constant on the (finite number of) atoms of $\mathcal{E}_k$ such that $\int \phi d \mathbb{P}=1$. Similarly let $B_k$ be the space of all simple functions on $\mathcal{E}_k$ with values in $[0,1]$. By taking our supremum over these smaller sets we have,
\begin{align*}
 \sup_{\phi \in A_k} \sup_{q \in B_k}  \inf_h& \Bigg\{\int   \left(\sum_{e \in U} \lb \lambda ,e \rb  - \log \frac{q( e)}{p(e)}+h-\mpt_e h\right)q(e) \phi d \mathbb{P}\Bigg\}
 \leq \Gamma(\lambda)
\end{align*}
We are now ready to apply the following minimax theorem due to Ky Fan~\cite{kyfan:53}.
\begin{theorem}
If $M$ is a compact subset of a topological space and the function $f: M \times N \mapsto \mathbb{R}$ is convex on $N$, concave on $M$ and upper semicontinuous on $M$ for every $\nu \in N$ then,
$$
\sup_{\mu \in M} \inf_{\nu \in N} f(\mu,\nu) =  \inf_{\nu \in N} \sup_{\mu \in M} f(\mu,\nu)
$$
\end{theorem}
Since the spaces  $A_k$, $B_k$ are compact  and the function to be maximized is continuous and concave in $\phi$ and $q$ and is convex in $h$ the theorem applies allowing us to switch the order of the infimum and supremum.
\begin{align*}
 \inf_h  \sup_{\phi \in A_k} \sup_{q \in B_k} & \Bigg\{\int   \left(\sum_{e \in U} \lb \lambda ,e \rb  - \log \frac{q( e)}{p(e)}+h-\mpt_e h\right)q(e) \phi d \mathbb{P}\Bigg\}
 \leq \Gamma(\lambda)
\end{align*}
which is equivalent to 
\begin{align*}
 \inf_h  \esup_{\omega} \sup_{q \in B_k} &\mathbb{E} \left[  \left(\sum_{e \in U} \lb \lambda ,e \rb  - \log \frac{q( e)}{p(e)}+h-\mpt_e h\right)q(e)\Big |\mathcal{E}_k \right]
 \leq \Gamma(\lambda)
\end{align*}
%-----------------------------------------------------------------------------------------------------------------------------------
Since $q$ in $B_k$ is $\mathcal{E}_k$-measurable we have,
\begin{align*}
 \inf_h  \esup_{\omega} \sup_{q \in B_k} \Big \{\sum_{e \in U} \Bigl (\lb \lambda, e \rb  - \log q(e)  +  \mathbb{E}\bigl [ \log p(e)  + h- \mpt_e h|\mathcal{E}_{k}\bigr] \Bigl) q(e)\Big\} \leq \Gamma(\lambda)
\end{align*}
The supremum over $q$ is an elementary calculus problem (see the appendix), resulting in
\begin{align*}
\inf_h  \esup_\omega \left\{  \log \sum_{e \in U} \base^{\lb \lambda, e \rb+\mathbb{E} \left[\log p(\omega,e)+h(\omega)-h(\mpt_e \omega)|\mathcal{E}_k \right]}  \right\}& \leq \Gamma(\lambda)\\
\end{align*}
for any $k\in (1,2, \ldots)$.
This means that for any $\epsilon >0, k \in \mathbb{N}$ there exists $h_{k,\epsilon}$ such that for $\mathbb{P}$-a.e.~$\omega$,
\begin{equation}\label{e}
\log \sum_{e \in U} \base^{\lb \lambda, e \rb+\mathbb{E} \left[\log p(\omega,e)+h_{k,\epsilon}(\omega)-h_{k,\epsilon}(\mpt_e \omega)|\mathcal{E}_k \right]}  \leq \Gamma(\lambda) + \epsilon
\end{equation}
We proceed with the construction of $F_{k,\epsilon}$. Define $H_{k,\epsilon}(\omega,e)=h_{k,\epsilon}(\omega)-h_{k,\epsilon}(\mpt_e \omega)$  and define $G_{k,\epsilon}(\omega,e)=\mathbb{E}\left [H_{k,\epsilon}(\omega,e)|\mathcal{E}_k \right]$ then set $F_{k,\epsilon}(\omega,e)=\mathbb{E}\left[ H_{k,\epsilon}(\omega,e)|\mathcal{E}_{k-1}\right ]$. 

Next we need to show that $\{F_{k,\epsilon}\}$ is weakly compact. Inequality (\ref{e}) shows that
$$
 G_{k,\epsilon}(\omega,e) \leq \Gamma(\lambda) -\lb \lambda,e \rb-\mathbb{E}\left [\log p(\omega,e) |\mathcal{E}_k \right] + \epsilon
$$ and by conditioning with respect to $\mathcal{E}_{k-1}$ that
$$
 F_{k,\epsilon}(\omega,e) \leq \Gamma(\lambda) -\lb \lambda,e \rb-\mathbb{E}\left [\log p(\omega,e) |\mathcal{E}_{k-1} \right] + \epsilon
$$

Note that $-H_{k,\epsilon}(\mpt_{-e}\omega,e)=-h_{k,\epsilon}(\mpt_{-e} \omega)+h_{k,\epsilon}(\omega)=H_{k,\epsilon}(\omega,-e)$. Therefore
\begin{align*}
\Gamma(\lambda) - \lb \lambda,  -e \rb &- \mathbb{E}\left[\log p(\omega,-e) |\mathcal{E}_k\right] + \epsilon \geq G_{k,\epsilon}(\omega,-e)\\
&=\mathbb{E}\left[H_{k,\epsilon}(\omega,-e)|\mathcal{E}_k\right] \\
&=\mathbb{E}\left[-H_{k,\epsilon}(\mpt_{-e}\omega,e)|\mathcal{E}_k\right]\\
&=\mathbb{E}\left[-H_{k,\epsilon}(\omega,e)|\mpt_e \mathcal{E}_k\right]\\
\end{align*} 
and since $\mpt_e \mathcal{E}_{k}\supset \mathcal{E}_{k-1}$, by conditioning with respect to $\mathcal{E}_{k-1}$ we have
$$
F_{k,\epsilon}(\omega,e) \geq -\big(\Gamma(\lambda) - \lb \lambda,-e \rb-\mathbb{E}\left[\log p(\omega,-e) |\mathcal{E}_{k-1}\right] + \epsilon\big)
$$
We conclude that,
\begin{align}\label{mk}
\left|F_{k,\epsilon}(\omega,e)\right| &\leq c_0 + \mathbb{E}\left[-\log p(\omega,e) |\mathcal{E}_{k-1}\right] + \mathbb{E}\left[-\log p(\omega,-e) |\mathcal{E}_{k-1}\right] \\
\intertext{and by Holder's inequality}
\left|F_{k,\epsilon}(\omega,e)\right|^{d+\alpha} &\leq c_1\left (1 + \mathbb{E}\left[-\log p(\omega,e) |\mathcal{E}_{k-1}\right]^{d+\alpha} + \mathbb{E}\left[-\log p(\omega,-e) |\mathcal{E}_{k-1}\right]^{d+\alpha} \right)\\
\intertext{Applying the conditional version of Jensen's inequality gives}
\left|F_{k,\epsilon}(\omega,e)\right|^{d+\alpha} & \leq c_1\left (1 + \mathbb{E}\left[|\log p(\omega,e)|^{d+\alpha} |\mathcal{E}_{k-1}\right] + \mathbb{E}\left[|\log p(\omega,-e)|^{d+\alpha} |\mathcal{E}_{k-1}\right] \right)\\
\intertext{and taking the expectation of both sides yields}
\mathbb{E} \left[ \left |F_{k,\epsilon}(\omega,e)\right|^{d+\alpha} \right] & \leq c_1 \left (1 + \mathbb{E} \left[| \log p(\omega,e)|^{d+\alpha} \right] + \mathbb{E} \left[ |\log p(\omega,-e)|^{d+\alpha} \right] \right)\\
\intertext{which gives} 
\label{mk2} \mathbb{E} \left[ \left |F_{k,\epsilon}(\omega,e)\right|^{d+\alpha} \right] & \leq c_2 < \infty
\end{align}
Therefore, the sequence $\{F_{k,\epsilon}(\omega,e)\}$ is weak $L^{d+\alpha}$-compact. We can therefore assume (passing to a subsequence if necessary) that $\{F_{k,\epsilon}(\omega,e)\}$ converges weakly to some function $F_\epsilon(\omega,e)$ in $L^{d+\alpha}$. 

We would like to have a strogly convergent subsequence of 
$$\{F_{k,\epsilon} + \mathbb{E}[\log p(\omega,e) | \mathcal{E}_{k-1}\}
$$ 
Set $\ell_k(\omega,e):=\mathbb{E}\left[\log p(\omega,e) |\mathcal{E}_{k-1}\right]$; then $\ell_k$ is a $L^{d+\alpha}$ bounded martingale. In particular, it converges strongly and therefore weakly in $L^{d+\alpha}$. Thus the function $S_{k,\epsilon}:=F_{k,\epsilon}+\ell_k$ converges weakly to $F_\epsilon+\log p$ in $L^{d+\alpha}$. We next use the following theorem of Mazur to construct our strongly convegent subsequence. The proof can be found in \cite{Rudin:91} page 67.
\begin{theorem}
Let $1 \leq p \leq \infty$. Suppose $\{m_n\}$ is a sequence in $L^p$ that converges weakly to $m$. Then there is a subsequence $\{M_n\}$ in $L^p$ that converges strongly to $m$ and such that for each $j$, $M_j$ is a convex combination of $\{m_1,m_2, \ldots , m_j\}$.
\end{theorem}
According to the theorem there exists a sequence $\{\tilde{S}_{k,\epsilon}(\omega,e)\}$ that converges stongly to $F_\epsilon(\omega,e)+\log p(\omega,e)$ in $L^{d+\alpha}$, where $\tilde{S}_{k,\epsilon}$ is a convex combination of $\{F_{j,\epsilon}+\ell_j\}_{j=1}^k$, say $\tilde{S}_{k,\epsilon}:=\sum_{j=1}^k a_j^k S_{j,\epsilon}$ and $\sum_{j=1}^k a_j^k =1$. Again by passing to a subsequence, we can assume that $\tilde{S}_{k,\epsilon}$ converges almost surely to $F_\epsilon+\log p$.
Starting with (\ref{e}) we have
\begin{align*}
 \exp\left\{\Gamma(\lambda)+\epsilon\right\} & \geq \sum_{e \in U}\exp \left \{\lb \lambda,e \rb+\mathbb{E}\left[\log p(\omega,e)|\mathcal{E}_j\right] +G_{j,\epsilon}(\omega,e)\right\}   \\ 
 \intertext{Conditioning with respect to $\mathcal{E}_{j-1}$ gives}
& \geq  \sum_{e \in U}\mathbb{E} \left [\exp \left\{\lb \lambda,e \rb+\mathbb{E}\left[\log p(\omega,e)|\mathcal{E}_{j-1}\right] +G_{j,\epsilon}(\omega,e)\right\} \big|\mathcal{E}_{j-1} \right]\\
\intertext{and by Jensen's inequality}
& \geq  \sum_{e \in U}\exp \left\{\lb \lambda,e \rb+\mathbb{E}\left[\log p(\omega,e)|\mathcal{E}_{j-1}\right] +F_{j,\epsilon}(\omega,e)\right\} \\
&=\sum_{e \in U}\exp \left\{\lb \lambda,e \rb+S_{j,\epsilon}\right\} \\
& \geq \sum_{j=1}^k a_j^k \sum_{e \in U}\exp \left\{\lb \lambda,e \rb+S_{j,\epsilon}\right\} \\
\intertext{and by the convexity of the exponential function, and Jensen's inequality once again,}
& \geq  \sum_{e \in U}\exp\left\{\lb \lambda,e \rb+ \sum_{j=1}^k a_j^k S_{j,\epsilon}\right\} \\
& =\sum_{e \in U}\exp\left\{\lb \lambda,e \rb+  \tilde{S}_{k,\epsilon}\right\} \\
\intertext{Since $\tilde{S}_{k,\epsilon}$ converges a.s. to $F_\epsilon +\log p$, letting $k \rightarrow \infty$ gives}
& =\sum_{e \in U}\exp\left\{\lb \lambda,e \rb+  \log p(\omega,e)+F_\epsilon(\omega,e) \right\} \\
&=   \sum_{e \in U} p(\omega,e) \base^{\lb \lambda,e \rb+F_\epsilon(\omega,e)} \\ 
\intertext{so that}
\Gamma(\lambda)+\epsilon &\geq \log  \sum_{e \in U} p(\omega,e) \base^{\lb \lambda,e \rb+F_\epsilon(\omega,e)}
\end{align*} 

To conclude the proof it remains to show that $F_\epsilon(\omega,e) \in \admiss$. Condition (i) that $\mathbb{E}[|F_\epsilon(\omega,e)|^{d+\alpha}] < \infty$ follows from (\ref{mk2}) and the fact that the $L^{d+\alpha}$-norm is weakly lower semicontinuous; indeed 
\begin{align*}
\mathbb{E}[|F_\epsilon(\omega,e)|^{d+\alpha}]^\frac{1}{d+\alpha}  &\leq \liminf_{n \rightarrow \infty} \mathbb{E}[|F_{n,\epsilon}(\omega,e)|^{d+\alpha}]^\frac{1}{d+\alpha}\\
& \leq c_2^\frac{1}{d+\alpha}\\
\mathbb{E}[|F_\epsilon(\omega,e)|^{d+\alpha}] &\leq c_2 < \infty
\end{align*}
The second condition (ii) that $\mathbb{E}[F_\epsilon(\omega,e)]=0$  follows immediately from the definition of weak covergence and the fact that the constant function $1$ is in $L^p(\mathbb{P})$ for any $p$. Indeed 
\begin{align*}
\mathbb{E}[F_{k,\epsilon}]&=\mathbb{E}[\mathbb{E}[H_{k,\epsilon}|\mathcal{E}_{k-1}]]\\
&=\mathbb{E}[\mathbb{E}[h_{k,\epsilon}(\omega)-h_{k,\epsilon}(\mpt_e \omega)|\mathcal{E}_{k-1}]]\\
&=\mathbb{E}[h_{k,\epsilon}(\omega)]-\mathbb{E}[h_{k,\epsilon}(\mpt_e \omega)]\\
&=0
\end{align*}
On the other hand weak convergence implies that
$$
\lim_{k \rightarrow \infty} \mathbb{E}[F_{k,\epsilon}]=\mathbb{E}[F_{\epsilon}]
$$
Finally, to prove (iii), suppose the sequence $\{x_i\}_{i=0}^n \in \mathbb{Z}^d$ is such that $x_0=x_n$ and $x_{i+1}-x_i \in U$. Then
\begin{align*}
\sum_{i=0}^{n-1} F_\epsilon(\mpt_{x_i}\omega, x_{i+1}-x_i)&= \sum_{i=0}^{n-1} F_\epsilon(\mpt_{x_i}\omega, x_{i+1}-x_i)\\
&= \wlim_{k \rightarrow \infty} \sum_{i=0}^{n-1} F_{k,\epsilon}(\mpt_{x_i}\omega, x_{i+1}-x_i)\\
&=\wlim_{k \rightarrow \infty} \sum_{i=0}^{n-1} \mathbb{E} \left [h_{k,\epsilon}(\mpt_{x_i}\omega)-h_{k,\epsilon}(\mpt_{x_{i+1}}\omega)\big | \mathcal{E}_{k-1}\right]\\
&=\wlim_{k \rightarrow \infty} \mathbb{E} \left [h_{k,\epsilon}(T_{x_0} \omega)-h_{k,\epsilon}(\mpt_{x_n}\omega)\big | \mathcal{E}_{k-1}\right]\\
&=0
\end{align*}
thus concluding the proof of Theorem \ref{uls} and that the upper and lower bounds are equivalent.
\end{proof}
%-----------------------------------------------------------------------------------------------------------------------------------
\chapter{Large Deviation Principle}
As usual we split the task of proving the large deviation principle into proving an upper bound and a lower bound.
\section{Upper Bound}
The upper bound can be proved using the G$\ddot{\mathrm{a}}$rtner-Ellis Theorem (see for example \cite{dembo:98}) since the domain of $\Lambda$ turns out to be all of $\mathbb{R}^d$. For completeness we prove the upper bound directly.
\begin{theorem} \label{ldpc}
For any closed set $C \in \mathbb{R}^d$
\begin{align*}
\limsup_{n \rightarrow \infty} \frac{1}{n} \log P_\omega \left[ \frac{X_n}{n} \in C\right]  & \leq - \inf_{x \in C} \sup_\lambda \{\lb \lambda, x \rb- \Lambda(\lambda)\}\\
\end{align*}
\end{theorem}
\begin{proof}
Since the walk is nearest neighbor, $|X_n|/n \leq 1$. Therefore,
$$
P_\omega \left[ \frac{X_n}{n} \in C\right]=P_\omega \left[ \frac{X_n}{n} \in C \cap B \right]
$$ 
where $B$ is the closed unit ball centered at the origin. Hence it suffices to prove the result for compact sets.
Set  $I(x)=\sup_\lambda \{\lb \lambda, x \rb- \Lambda(\lambda)\}$ and let $I_\delta=\min \{I(x)-\delta,1/\delta \}$. Let $K$ be a compact set in $\mathbb{R}^d$. By the definition of $I_\delta$ we may choose for every $z \in K$, $\lambda_z \in \mathbb{R}^d$ so that
$$
\lb \lambda_z, z \rb - \Lambda(\lambda_z) \geq I_\delta (z)
$$
For each $z \in K$ choose $\alpha_z >0$, so that $\alpha_z |\lambda_z| \leq \delta$. Define $D_z$ to be the open ball centered at $z$ with radius $\alpha_z$ that is $D_z = \{ x : |x-z| < \alpha_z\}$. Then by Chebyshev's inequality,
$$
P_\omega \left [X_n/n \in D_z \right] \leq E^{P_\omega}\left[ \base^{\lb \lambda_z, X_n \rb} \right] \base^{-\inf_{x \in D_z}\{n\lb\lambda_z,x\rb\}} 
$$
Since $-\lb \lambda_z,x \rb = \lb \lambda_z,z-x \rb - \lb \lambda_z,z \rb$, we know that
\begin{align*}
- \inf_{x \in D_z} \lb \lambda_z,x \rb & \leq \alpha_z |\lambda_z|-\lb \lambda_z,z \rb\\
& \leq \delta -\lb \lambda_z,z \rb
\end{align*}
This gives us,
$$
\frac{1}{n} \log P_\omega \left [X_n/n \in D_z \right] \leq E^{P_\omega}\left[ \base^{\lb \lambda_z, X_n \rb} \right] + \delta - \lb \lambda_z, z \rb
$$
By the compactness of $K$ we can choose a finite number of these balls, say $N$,  centered at points $z_i$ that cover $K$. Then
$$
\frac{1}{n} \log P_\omega \left [X_n/n \in K \right] \leq \frac{1}{n} \log N + \delta - \min_{i=1,\ldots,N}\left\{ \lb \lambda_{z_i}, z_i \rb -\frac{1}{n} \log E^{P_\omega}\left[ \base^{\lb \lambda_{z_i}, X_n \rb} \right] \right\}
$$
and hence,
\begin{align*}
\limsup_{n \rightarrow \infty}\frac{1}{n} \log P_\omega \left [X_n/n \in K \right]  & \leq  \delta - \min_{i=1,\ldots,N}\left\{ \lb \lambda_{z_i}, z_i \rb -\Lambda(\lambda_{z_i})\right\}\\
\intertext{and by our choice of the $\lambda_{z_i}$}
& \leq  \delta - \min_{i=1,\ldots,N} I_\delta (\lambda_{z_i})\\
& \leq \delta - I_\delta (K)
\end{align*}
where $I(K)=\inf_{x \in K} I(x)$. Letting $\delta \downarrow 0$ gives the upper bound.
\end{proof}
%-----------------------------------------------------------------------------------------------------------------------------------
\section{Lower Bound}

\begin{theorem} \label{ldpo}
For any open set $O\in \mathbb{R}^d$
\begin{align*}
\liminf_{n \rightarrow \infty} \frac{1}{n} \log P_\omega \left[ \frac{X_n}{n} \in O\right]  & \geq - \inf_{x \in O} \sup_\lambda \{\lb \lambda, x \rb- \Lambda(\lambda)\}\\
\end{align*}
\end{theorem}
\begin{proof}
It suffices to prove that for any ball $D_{\epsilon}(x)$  (an open ball centered at $x$ with radius $\epsilon$) that
\begin{align} \label{bex}
\liminf_{n \rightarrow \infty} \frac{1}{n} \log P_\omega \left[ \frac{X_n}{n} \in D_\epsilon (x) \right]  & \geq -  \sup_\lambda \{\lb \lambda, x \rb- \Lambda(\lambda)\}
\end{align}
This is because for each point $x$ in an open set $O$ there is a number $\epsilon$ such that $D_\epsilon (x) \subset O$. So that after taking the infimum over $x \in O$ of both sides of (\ref{bex}) we see that for each $x \in O$ and such an $\epsilon$ the left hand side of ($\ref{bex}$) is less than or equal to 
$$
\liminf_{n \rightarrow \infty} \frac{1}{n} \log P_\omega \left[ \frac{X_n}{n} \in O\right] 
$$
Hence we can proceed to estimate,
\begin{align*}
\liminf_{n \rightarrow \infty} \frac{1}{n} \log P_\omega\left[ \frac{X_n}{n}\in D_{\epsilon} (x) \right]&=\liminf_{n \rightarrow \infty} \frac{1}{n} \log E^{P_\omega}\left[ 1_{D_{\epsilon} (x)}\left(\frac{X_n}{n}\right) \right]\\
&=\liminf_{n \rightarrow \infty} \frac{1}{n} \log E^{Q_\omega}\left[ 1_{D_{\epsilon} (x)}\left(\frac{X_n}{n}\right) \frac{dP_\omega}{dQ_\omega}\right]
\end{align*}  Consider the probability measure $R_\omega$ defined by 
$$
dR_\omega = \frac{1_{D_\epsilon (x)}(X_n/n)} {Q_\omega (D_\epsilon (X_n/n))}dQ_\omega
$$
setting $D=D_\epsilon(x)$ we calculate
\begin{align}\label{jen}
\log E^{Q_\omega}\left[ 1_{D}(X_n/n) \frac{dP_\omega}{dQ_\omega}\right]&= \log \left\{Q_\omega (D) E^{R_\omega} \left[\frac{dP_\omega}{dQ_\omega}\right] \right\}\\
&= \log Q_\omega(D) + \log E^{R_\omega}\left [\frac{dP_\omega}{dQ_\omega}\right]\\
\intertext{and by Jensen's inequality}
&\geq \log Q_\omega (D)-\frac{1}{Q_\omega (D)}E^{Q_\omega}\left[1_D(X_n/n) \log \frac{dQ_\omega}{dP_\omega}\right]\\
\intertext{multiplying by $1/n$ gives us}
\frac{1}{n}\log E^{Q_\omega}\left[ 1_{D}(X_n/n) \frac{dP_\omega}{dQ_\omega}\right] & \geq \frac{\log Q_\omega (D)}{n}-\frac{1}{Q_\omega (D)}E^{Q_\omega}\left[ \frac{1_{D}(X_n/n)}{n} \log \frac{dQ_\omega}{dP_\omega}\right]
\end{align}
If $d\mathbb{Q}=\phi(\omega)d\mathbb{P}$ is an ergodic invariant measure for the $q$-markov chain such that 
$$
\int q(\omega,e)-q(\omega,-e)\phi(\omega)d\mathbb{P}=\lb x,e \rb
$$then by (\ref{dqdp}) and the law of large numbers for the RWRE the right hands side of (\ref{jen})  tends to
$$
-\int_\Omega \sum_{e \in U} \log \frac{q(\omega,e)}{p(\omega,e)}q(\omega,e) \phi(\omega) d \mathbb{P}
$$ as $n$ tends to infinity. Therefore
$$
\liminf_{n \rightarrow \infty} \frac{1}{n} \log P_\omega\left[ \frac{X_n}{n}\in D_{\epsilon} (x) \right] \geq
-g(x)
$$ 
where
\begin{equation} \label{gx}
g(x):=\inf_{(q,\phi) \in A_x}\int_\Omega \sum_{e \in U} \log \frac{q(\omega,e)}{p(\omega,e)}q(\omega,e) \phi(\omega) d \mathbb{P}
\end{equation}
$A_x$ is the collection of pairs $(q, \phi)$,  $\phi$ is an the egodic invariant distribution for $q$ and $\int q(e)-q(-e)\phi(\omega)d\mathbb{P}=\lb x,e \rb$.

We know from the conclusion of Theorem \ref{a} that
\begin{align*}
\lim_{n \rightarrow \infty} \frac{1}{n} \log E^{P_\omega}\left[ \base^{\lb \lambda, X_n \rb}\right] 
&=\sup_{(q,\phi) \in A}\left\{\int_\Omega \lb \lambda, e \rb q(e)-\sum_{e \in U} \log \frac{q(e)}{p(e)}q(e)  \phi d \mathbb{P} \right\}\\
&=\sup_x \sup_{(q,\phi) \in A_x}\left\{\lb \lambda, x \rb-\int_\Omega \sum_{e \in U} \log \frac{q(e)}{p(e)}q(e)\phi d \mathbb{P}\right\}\\
&= \sup_x \left\{\lb \lambda, x \rb-\inf_{(q,\phi) \in A_x}\int_\Omega \sum_{e \in U} \log \frac{q(e)}{p(e)}q(e) \phi d \mathbb{P}\right\}\\
&=\sup_x \left\{\lb \lambda, x \rb-g(x)\right\}
\end{align*}
We show below in Lemma \ref{gcon} that $g(x)$ is convex, identifying it as the convex conjugate of $\Lambda(\lambda)$ thus concluding the proof.
\end{proof}
\begin{lemma} \label{gcon}
The function $g(x)$  in (\ref{gx}) is convex.
\end{lemma}
\begin{proof}
Set 
$$
\gamma(q,\phi):=\int_\Omega\sum_{e \in U} \log \frac{q(\omega,e)}{p(\omega,e)}q(\omega,e) \phi(\omega) d \mathbb{P}
$$
Then 
$$
g(x)= \inf_{(q,\phi) \in A_x} \gamma(q,\phi)
$$
We want to show that for $a \in [0,1]$, $b=1-a$, that
$$
g(ax^1+bx^2) \leq ag(x^1) +bg(x^2)
$$
The definition of $g(x)$ guarantee's that for a given $\epsilon >0$, we can choose $(q_k,\phi_k) \in A_{x^k}$ so that $\gamma(q_k,\phi_k) \leq g(x^k)+\epsilon/2$. Assume that we have chosen such a $(\phi_k,q_k) \in A_{x^k}$ for $k=1,2$.
Our aim is to construct $q_3$, $\phi_3$ so that 
\begin{enumerate}
\item[(i)] $\phi_3$ is an ergodic invariant distribution for $q_3$
\item[(ii)] $\int (q_3(\omega,e)-q_3(\omega,-e)\phi_3 d\mathbb{P}=a \lb x^1,e \rb +b\lb  x^2,e
\rb$
\item[(iii)] $\gamma(q_3,\phi_3) \leq a \gamma(q_1,\phi_1)+ b \gamma(q_2,\phi_2) $
\end{enumerate}
Then  $q_3$, $\phi_3$ will be in $A_{ax^1+bx^2}$ and 
\begin{align*} 
g(ax^1+bx^2) &\leq \gamma(q_3,\phi_3) \leq a \gamma(q_1,\phi_1)+ b \gamma(q_2,\phi_2) \\
&\leq ag(x^1)+bg(x^2)+\epsilon
\end{align*}
proving the lemma.
To construct $(q_3, \phi_3)$ define,
\begin{align*}
\alpha &:=  \frac{a \phi_1}{a \phi_1 + b \phi_2}\\
\beta &:= \frac{b \phi_2}{a \phi_1 +b  \phi_2}\\
q_3(\omega,e) &:=\alpha q_1(\omega,e) +\beta q_2(\omega,e)\\
\phi_3 &:= a \phi_1 + b \phi_2 \\
\end{align*}
To check condition (i), we have to show that  $\phi_3$ is an ergodic invariant distribution for $q_3$. We use the condition (\ref{stat})  and calculate that for any bounded measurable function $h$,
\begin{align*}
\int \sum_{e \in U} h(\mpt_e \omega) q_3\phi_3 d\mathbb{P}
&=\int \sum_{e \in U}  h(\mpt_e \omega) (\alpha q_1 +\beta q_2 )(a \phi_1 + b \phi_2) d\mathbb{P}\\
&=a \int   \sum_{e \in U}  h(\mpt_e \omega) q_1  \phi_1  d\mathbb{P} + b \int  \sum_{e \in U}   h(\mpt_e \omega) q_2  \phi_2  d\mathbb{P} \\
&= a \int h \phi_1 d \mathbb{P} +b \int h \phi_2 d \mathbb{P}\\
&= \int h \phi_3 d \mathbb{P} \\
\end{align*}
Showing that (i) is indeed true.

To check condition (ii) we simply expand the definitions to compute
\begin{align*}
\int (q_3(e)-q_3(-e)) \phi_3 d\mathbb{P} &=\int (\alpha q_1(e) +\beta q_2(e) -\alpha  q_1(-e) -\beta q_2(-e))(a \phi_1 +b \phi_2) d \mathbb{P}\\
\begin{split}
&= \int \alpha (q_1(e)-q_1(-e))(a \phi_1 +b \phi_2) d \mathbb{P}\\
& + \int \beta (q_2(e)-q_2(-e))(a \phi_1 +b \phi_2) d \mathbb{P}
\end{split}\\
&= a\int  (q_1(e)-q_1(-e)) \phi_1 d \mathbb{P} +b \int  (q_2(e)-q_2(-e)) \phi_2 d \mathbb{P}\\
&=a \lb x^1,e \rb+b \lb x^2,e \rb
\end{align*}

Lastly, we prove condition (iii), We have,
\begin{align*}
\gamma(q_3,\phi_3) &= \int \left\{ \log \frac{ q_3(e)}{p(e)} q_3(e)+\log \frac{q_3(-e)}{p(-e)}q_3(-e)\right\} \phi_3 d \mathbb{P}\\
&= \int  \log \frac{\alpha q_1(e)+\beta q_2(e)}{p(e)}(\alpha q_1(e)+\beta q_2(e)) \phi_3 d \mathbb{P}\\
&+ \int \log \frac{\alpha q_1(-e)+\beta q_2(-e)}{p(-e)}(\alpha q_1(-e)+\beta q_2(-e)) \phi_3 d\mathbb{P}\\
\intertext{and since $x \log x$ is convex, by Jensen's inequality}
& \leq  \int  \left\{\alpha \log \frac{ q_1(e)}{p(e)} q_1(e)+\alpha\log \frac{q_1(-e)}{p(-e)}q_1(-e)\right\} \phi_3 d \mathbb{P}\\
&+  \int  \left\{\beta \log \frac{ q_2(e)}{p(e)} q_2(e)+\beta \log \frac{q_2(-e)}{p(-e)}q_2(-e)\right\} \phi_3 d \mathbb{P}\\
&=\int  \left\{a \log \frac{ q_1(e)}{p(e)} q_1(e)\phi_1+a \log \frac{q_1(-e)}{p(-e)}q_1(-e)\phi_1\right\}  d \mathbb{P}\\
&+\int  \left\{ b \log \frac{ q_2(e)}{p(e)} q_2(e)\phi_2+b \log \frac{q_2(-e)}{p(-e)}q_2(-e)\phi_2\right\}  d \mathbb{P}\\
&= a \gamma(q_1,\phi_1)+ b \gamma(q_2,\phi_2)\\
\end{align*}
proving that $g(x)$ is indeed convex.
\end{proof} 
 \chapter{The One Dimensional Case}
 \section{Previous results}
 The situation in one dimension has been well understood for some time. A good starting place for a more detailed study and a more comprehensive bibliography is  \cite{zeitouni:03}. Here we will show the equivalence of our results with those perviously obtained in the  one dimensional case .
  
  Let $\tau_1= \inf \{n: X_n=1 \}$ and   $\tau_{-1}= \inf \{n: X_n=-1 \}$ and for any $r \in \mathbb{R}$,
\begin{align*}
 G(\omega,r)&=E^{P_\omega}[\base^{r \tau_1} 1_{\{\tau_1 < \infty\}}]; & H(\omega,r)&=E^{P_\omega}[\base^{r \tau_{-1}} 1_{\{\tau_{-1} < \infty\}}] \\
 g (r)&=\mathbb{E}[\log G(\omega, r)]; &  h (r)&=\mathbb{eE}[\log H(\omega, r)]
 \end{align*}
 Comets, Gantert, and Zeitouni \cite{comets:00} present the following quenched large deviation principle for an ergodic nearest neighbor random walk in a random environment,
  \begin{theorem}
 Assume that $(\mathbb{P},\mpt) $ is ergodic and that the random walk is uniformly elliptic. Further assume that $\mathbb{E}[\log (p^-/p^+)] \leq 0$ (i.e. the walk is transient to the right). Then  $X_n/n$ satisfies a large deviation principle with rate function,
 \begin{equation}
 J(x)=
 \begin{cases}
 \sup_{r \in \mathbb{R}} \left \{r-x g(r)] \right \} & 0 \leq x \leq 1 \\
 \sup_{r \in \mathbb{R}} \left \{r+x h(r)] \right \} & -1 \leq x \leq 0
 \end{cases}
 \end{equation}
   \end{theorem}
 \begin{remark} 
 There is an analogous result for the case where the walk is transient to the left.
 \end{remark}
 \section{Our result in one dimension}
In one dimension, $U=\{-1,+1\}$. We will use the following shorthand to make the notation more transparent.
 \begin{align*}
 F(\omega,+1) &=F^{+}(\omega) \\
 F(\omega,-1) &=F^{-}(\omega)\\
 p(\omega,+1) &=p^{+}(\omega) \\
 p(\omega,-1) &=p^{-}(\omega)
 \end{align*}
 The class $\admiss$ includes those pairs of functions $F^{\pm}$ such that
 \begin{enumerate}
 \item[(i)] Moment: $F \in  L^1(\mathbb{P})$. 
  \item[(ii)] Mean Zero:  $\mathbb{E}[F^{\pm}(\omega)]=0$.
  \item[(iii)] Closed Loop: $F^{+}(\omega) + F^{-}(\mpt \omega)=0$
 \end{enumerate} 
 \begin{remark}
Notice that $F$ only needs to be in $L^1(\mathbb{P})$ as opposed to $\bigcup_{\alpha>0} L^{1+\alpha}(\mathbb{P})$.  Looking back at the parts of  the proof of Theorem \ref{a} where $L^{d+\alpha}$ was needed will show that in one dimension $L^1$ will suffice.
\end{remark}
The one dimensional version of Definition \ref{lam} and Theorems \ref{a} and \ref{b} are:
\begin{definition} 
$$
\Lambda(\lambda) := \inf_{F \in \admiss} \esup_\omega \log \left \{ p^{+}(\omega) \base^{ \lambda+F^{+}(\omega)}+p^{-}(\omega) \base^{- \lambda+F^{-}(\omega)}\right\}
$$
where the $\esup$ is with respect to the measure $\mathbb{P}$.
\end{definition}
and
\begin{theorem} \label{a1}
Suppose $\int \left |\log p^{\pm}(\omega)\right| d\mathbb{P}<\infty$. Then
\begin{equation}
\begin{split}
 \lim _{n \rightarrow \infty} \frac{1}{n} \log E^{P_\omega}\left [\base^{\lambda X_n }  \right]&=\Lambda(\lambda)
 \end{split}
 \end{equation}
\end{theorem}
\begin{theorem} 
Under the assumptions of Theorem \ref{a1}, 
$X_n/n$ obeys a large deviation principle with rate function
\begin{equation}
I(x)=\sup_\lambda \left\{ \lambda x  - \Lambda(\lambda) \right \}
\end{equation}
\end{theorem}

 \section{Equivalence with Previous Results}
 In this section we will show that the rate functions $I(x)$ and $J(x)$ are in fact equal. In order to facilitate this it is easiest to put the rate function $I(x)$ in a slightly different form. We start with some definitions.
 \begin{definition}
 The set $A$ is made up of all pairs $(\theta, \lambda) \in \mathbb{R} \times \mathbb{R}$ such that there exists $F^{\pm} \in \admiss$ such that
 $$
 \log \{p^+\base^{\theta +F^+} + p^- \base^{-\theta + F^-}\} \leq \lambda
 $$
 \end{definition}
 \begin{definition}
 $$
 \tilde{I}(x)=\sup_{(\theta,\lambda) \in A} \{ \theta x - \lambda\}
 $$
 \end{definition}
 \begin{lemma}
 $$
  \tilde{I}(x)=I(x)
 $$
 \end{lemma}
 \begin{proof}
 We have $I(x)=\sup_\theta \{ \theta x - \Lambda(\lambda)\}$. Now if $(\theta, \lambda) \in A$, then there exists $F$ such that $ \log \{p^+\base^{\theta +F^+} + p^- \base^{-\theta + F^-}\} \leq \lambda $. The definition of $\Lambda$ is
 $$
 \Lambda(\theta) = \inf_{F \in \admiss} \esup_\omega \log \{p^+\base^{\theta +F^+} + p^- \base^{-\theta + F^-}\}
 $$
 that is, for all $F \in \admiss$, $\Lambda(\theta) \leq  \log \{p^+\base^{\theta +F^+} + p^- \base^{-\theta + F^-}\}$
so that for $(\theta,\lambda) \in A$ 
$$ 
\lambda \geq \Lambda(\theta)
$$ and hence
$$
\tilde{I}(x) \leq I(x)
$$
On the other hand, by the definition of $I(x)$, given $\epsilon >0$, for all $\theta$ there is an $F \in \admiss$ such that
$$
\log \{p^+\base^{\theta +F^+} + p^- \base^{-\theta + F^-}\} \leq \Lambda(\theta) + \epsilon
$$
which tells us that $(\theta, \Lambda(\theta)+\epsilon)$ is in $A$ therefore,
$$
\tilde{I}(x) \geq \sup_\theta \{ \theta x - (\Lambda(\theta)+\epsilon)) \}
$$
Letting $\epsilon$ tend to zero finishes the proof.
 \end{proof}
 We now turn to the main result of this section,
 \begin{theorem}
 I(x)=J(x)
 \end{theorem}
 \begin{proof}
 The first step is to find two examples of  members of the set $A$ in terms of the functions $G(\omega,r)$ and $H(\omega,r)$. We decompose according to the first step of the random walk to observe that $G(\omega,r)$ satisfies,
\begin{align} 
G(\omega,r)&=p^+ \base^r+p^- \base^r G(\mpt^{-1} \omega,r) G(\omega,r) \\
\base^{-r} &= p^+ \base^{-\log G(\omega,r)}+p^- \base^{\log G(\mpt^{-1} \omega,r)} \label{r}
\end{align}
We choose,
\begin{align*}
F_g^+ &= -\log G(\omega,r)-\theta_g \\
F_g^- &= \log G(\mpt^{-1} \omega,r) + \theta_g\\
\theta_g &= -g(r) \\
\lambda &= -r
\end{align*}
One easily checks that $F_g \in \admiss$ and that
$$
\log \{p^+\base^{\theta_g +F_g^+} + p^- \base^{-\theta_g + F_g^-}\} =-r
$$ 
so that indeed
\begin{equation}
(-g(r),-r) \in A \label{gA}
\end{equation}
Following the same procedure using the function $H$ we see that
$$
\base^{-r} = p^+ \base^{\log H(\mpt \omega,r)}+p^- \base^{-\log H( \omega,r)}
$$
Now we choose,
\begin{align*}
F_h^+ &= \log H(\mpt \omega,r)-\theta_h \\
F_h^- &= -\log H( \omega,r) + \theta_h\\
\theta_h &= h(r) \\
\lambda &= -r
\end{align*}
and again $F_h \in \admiss$ and
\begin{equation}
(h(r),-r) \in A
\end{equation}
Having found the two elements of $A$ that we require, the next step in the proof is to utilize the defining property of set $A$ to derive a relationship between $\theta, \lambda$ and $g$. Recall that $F(x,x+e)$ is defined to be equal to $F(\mpt_x \omega,e)$, which in our one dimensional case states that $F(x,x+1)=F^+(\mpt_x \omega)$ and $F(x,x-1) = F^-(\mpt_x \omega)$.
Set
$$
R_n = \exp \left \{\theta X_n + \sum_{j=1}^n F(X_{n-1},X_n)-n\lambda \right \}
$$
If $(\theta, \lambda) \in A$ then according to Lemma \ref{smart}, $\{R_n\}$ is a supermartingale with respect to $\sigma(X_0,X_1, \ldots X_n)$ and $R_0=1$. Hence, $E^{P_\omega}[R_n] \leq 1$ and in particular, since $\{R_n\}$ is a positive supermartingale the stopping theorem gives,
\begin{align*}
E^{P_\omega}[R_{\tau_1}1_{\{\tau_1 < \infty\}}]&\leq 1\\
\intertext{and since on $\{\omega: \tau_1 < \infty\}$ we have $R_{\tau_1}=\exp\{\theta +F^+(\omega) -\lambda \tau_1\}$ we have}
E^{P_\omega}[\exp\{\theta +F^+(\omega) -\lambda \tau_1\}1_{\{\tau_1 < \infty\}}]] & \leq 1  \\
\intertext{ so that}
E^{P_\omega}[\base^{-\lambda \tau_1}1_{\{\tau_1 < \infty\}}]] &\leq \base^{-\theta-F^+(\omega)}  \\
\intertext{Taking the logarithm followed by the expectation with respect to the environment gives,}
\mathbb{E}[\log E^{P_\omega}[\base^{-\lambda \tau_1}]1_{\{\tau_1 < \infty\}}]] & \leq \mathbb{E}[-\theta-F^+(\omega)] \\
\intertext{in other words,}
\theta &\leq -g(-\lambda)\\
\intertext{A similar argument will give us,}
\theta &\geq h(-\lambda)
\end{align*}
We deal first with the case when the argument of the rate function is positive, that is $0 \leq x \leq 1$. Since (\ref{gA}) gives us $(-g(\lambda),-\lambda) \in A$ we obtain,
\begin{align*}
I(x)=\tilde{I}(x) &\geq \sup_\lambda \{-g(\lambda)x + \lambda\}=J(x)
\end{align*}
On the other hand, since $-g(-\lambda) \geq \theta$, for $(\theta, \lambda) \in A$ and $x \geq 0$,
\begin{align*}
I(x) &\leq \sup_\lambda \{-g(-\lambda) x -\lambda\} \\
 & = \sup_\lambda \{-g(\lambda) x+ \lambda\} = J(x)
\end{align*}
For the case $-1 \leq x \leq 0$, since $(h(\lambda),-\lambda) \in A$
$$
I(x) \geq \sup_\lambda \{ h(\lambda)x + \lambda\}=J(x)
$$
and since $h(-\lambda) \leq \theta$, and $x \leq 0$,
\begin{align*}
I(x) &\leq \sup_\lambda \{ h(-\lambda) x - \lambda\}\\
& = \sup_\lambda \{ h(\lambda) x + \lambda\} = J(x)
\end{align*}
\end{proof}

%\include{conclusion}

%%-------------------------------------Appendix or Appendices

\appendix

\appheading{A Calculus Result}
We solve the following maximization problem.
\begin{align} \label{obj}
\sup_{q \in B_k} \Big \{\sum_{e \in U} \Bigl (\lb \lambda, e \rb  - \log q(e)  +  \mathbb{E}\bigl [ \log p(e)  + h- \mpt_e h|\mathcal{E}_{k}\bigr] \Bigl) q(e)\Big\}
\end{align}
Where $h$ is a bounded measurable function and $\mathcal{E}_k$ and $B_k$ are defined as in Chapter 3. Set $\nu(e)=\mathbb{E}\bigl [ \log p(e)  + h- \mpt_e h|\mathcal{E}_{k}\bigr] $. Since $q(e)$ and $\nu(e)$ are constant on the atoms of $\mathcal{E}_k$, we can treat $q$ and $\nu$ as constant functions of $\omega$. We calculate the first and second order conditions.
\begin{align*}
\frac{\partial}{\partial q(e)} \Big \{\sum_{e \in U} \Bigl (\lb \lambda, e \rb  - \log q(e)  +  \nu(e) \Bigl) q(e)\Big\}&=(\lb \lambda, e \rb-\log q(e) -1\\
\frac{\partial^2}{\partial q(e)^2} \Big \{\sum_{e \in U} \Bigl (\lb \lambda, e \rb  - \log q(e)  +  \nu(e) \Bigl) q(e)\Big\}&=-\frac{1}{q(e)}
\end{align*}
Since $q(e)>0$ the second order conditions will guarantee that we indeed have a maximum. Setting the first order conditions to zero implies that $q(e)$ is proportional to
$$
\exp(\lb \lambda, e \rb- \nu(e)-1)
$$
and since 
$$
\sum_{e \in U} q(e) =1
$$
The optimum of (\ref{obj}) is attained at 
$$
q^*(e)=\frac{\exp(\lb \lambda, e \rb- \nu(e))}{\sum_{e \in U}\exp(\lb \lambda, e \rb- \nu(e))}
$$
Substituting this value back into the objective function (\ref{obj}) we obtain
\begin{align*}
\sup_{q \in B_k} \Big \{\sum_{e \in U} \Bigl (\lb \lambda, e \rb  -& \log q(e)  +  \mathbb{E}\bigl [ \log p(e)  + h- \mpt_e h|\mathcal{E}_{k}\bigr] \Bigl) q(e)\Big\}\\
&=\log \sum_{e \in U} \exp(\lb \lambda, e \rb+ \nu(e))\\
&=\log \sum_{e \in U} \exp\Bigr(\lb \lambda, e \rb +  \mathbb{E}\bigl [ \log p(e)  + h- \mpt_e h|\mathcal{E}_{k}\bigr] \Bigl)
\end{align*}
%\include{appendixB}

%%----------------------------BIBLIOGRAPHY

\newpage

\addcontentsline{toc}{chapter}{Bibliography}
\nocite{*}
\bibliography{main}

\begin{thebibliography}{10}

\bibitem{comets:00}
F.~Comets, N.~Gantert, and O.~Zeitouni.
\newblock Quenched, annealed and functional large deviations for one
  dimensional random walk in random environment.
\newblock {\em Probab. Theory Related Fields}, 118:65--114, 2000.

\bibitem{dembo:98}
A.~Dembo and O.~Zeitouni.
\newblock {\em Large deviations techniques and applications}.
\newblock Springer, New York, second edition, 1998.

\bibitem{kyfan:53}
K.~Fan.
\newblock Minimax theorems.
\newblock {\em Proceedings of the National Academy of Sciences}, 39:42--47,
  1953.

\bibitem{Greven:94}
A.~Greven and F.~D. Hollander.
\newblock Large deviations for a random walk in random environment.
\newblock {\em Ann. Probab}, 22:1381--1428, 1994.

\bibitem{Kallenbert:02}
O.~Kallenberg.
\newblock {\em Foundations of Modern Probability}.
\newblock Springer-Verlag, New York, second edition, 2002.

\bibitem{Rudin:91}
W.~Rudin.
\newblock {\em Functinal Analysis}.
\newblock McGraw-Hill, second edition, 1991.

\bibitem{schroeder:88}
C.~Schroeder.
\newblock Green's functions for the schrodinger operator with periodic
  potential.
\newblock {\em Journal of functional analysis}, 77:60--87, 1988.

\bibitem{stroock:79}
D.~W. Stroock and S.~R.~S. Varadhan.
\newblock {\em Multidimensional Diffusion Processes}.
\newblock Springer-Verlag, Berlin, corrected second printing edition, 1997.

\bibitem{sznitman:02}
A.-S. Sznitman and E.~Bolthausen.
\newblock {\em Ten Lectures on Random Media}, volume~32 of {\em DMV Seminar}.
\newblock Birkhauser, 2002.

\bibitem{zeitouni:03}
S.~Tavare and O.~Zeitouni.
\newblock {\em Lectures on Probability Theory and Statistics}.
\newblock Springer, 2004.

\bibitem{varad:03}
S.~R.~S. Varadhan.
\newblock Large deviations for random walks in a random environment.
\newblock {\em Comm. Pure Appl. Math.}, 56:1222--1245, 2003.

\bibitem{zerner:98}
M.~P.~W. Zerner.
\newblock Lyapounov exponents and quenched large deviations for
  multidimensional random walk in random environment.
\newblock {\em Annals of Probability}, 26:1446--1476, 1998.

\end{thebibliography}
\end{document}